\documentclass[12pt,reqno]{amsart}

\newtheorem{thm}{Theorem}
\newtheorem{defn}[thm]{Definition}
\newtheorem{lem}[thm]{Lemma}

\newtheorem{conj}[thm]{Conjecture}

\begin{document}

\title{The road to deterministic matrices with the restricted isometry property}

\author[Bandeira]{Afonso S.~Bandeira}
\address[Bandeira]{Program in Applied and Computational Mathematics, Princeton University, Princeton, New Jersey 08544, USA; E-mail: ajsb@math.princeton.edu}

\author[Fickus]{Matthew Fickus}
\address[Fickus]{Department of Mathematics and Statistics, Air Force Institute of Technology, Wright-Patterson Air Force Base, OH 45433, USA; matthew.fickus@afit.edu}

\author[Mixon]{Dustin G.~Mixon}
\address[Mixon]{Program in Applied and Computational Mathematics, Princeton University, Princeton, New Jersey 08544, USA; E-mail: dmixon@princeton.edu}

\author[Wong]{Percy Wong}
\address[Wong]{Program in Applied and Computational Mathematics, Princeton University, Princeton, New Jersey 08544, USA; E-mail: pakwong@math.princeton.edu}

\begin{abstract}
The restricted isometry property (RIP) is a well-known matrix condition that provides state-of-the-art reconstruction guarantees for compressed sensing.
While random matrices are known to satisfy this property with high probability, deterministic constructions have found less success.
In this paper, we consider various techniques for demonstrating RIP deterministically, some popular and some novel, and we evaluate their performance.
In evaluating some techniques, we apply random matrix theory and inadvertently find a simple alternative proof that certain random matrices are RIP.
Later, we propose a particular class of matrices as candidates for being RIP, namely, equiangular tight frames (ETFs).
Using the known correspondence between real ETFs and strongly regular graphs, we investigate certain combinatorial implications of a real ETF being RIP.
Specifically, we give probabilistic intuition for a new bound on the clique number of Paley graphs of prime order, and we conjecture that the corresponding ETFs are RIP in a manner similar to random matrices.
\end{abstract}



\thanks{The authors thank Prof.~Peter Sarnak and Joel Moreira for insightful discussions and helpful suggestions.
ASB was supported by NSF Grant No.~DMS-0914892, MF was supported by NSF Grant No.~DMS-1042701 and AFOSR Grant Nos.~F1ATA01103J001 and F1ATA00183G003, and DGM was supported by the A.B.~Krongard Fellowship.
The views expressed in this article are those of the authors and do not reflect the official policy or position of the United States Air Force, Department of Defense, or the U.S. Government.}

\maketitle

\section{Introduction}

Let $x$ be an unknown $N$-dimensional vector with the property that at most $K$ of its entries are nonzero, that is, $x$ is $K$\emph{-sparse}.
The goal of compressed sensing is to construct relatively few non-adaptive linear measurements along with a stable and efficient reconstruction algorithm that exploits this sparsity structure.
Expressing each measurement as a row of an $M\times N$ matrix $\Phi$, we have the following noisy system:
\begin{equation}
\label{eq.cs eqn}
y=\Phi x+z.
\end{equation}
In the spirit of \emph{compressed} sensing, we only want a few measurements: $M\ll N$.
Also, in order for there to exist an inversion process for \eqref{eq.cs eqn}, $\Phi$ must map $K$-sparse vectors injectively, or equivalently, every subcollection of $2K$ columns of $\Phi$ must be linearly independent.
Unfortunately, the natural reconstruction method in this general case, i.e., finding the sparsest approximation of $y$ from the dictionary of columns of $\Phi$, is known to be NP-hard~\cite{Natarajan:95}.
Moreover, the independence requirement does not impose any sort of dissimilarity between columns of $\Phi$, meaning distinct identity basis elements could lead to similar measurements, thereby bringing instability in reconstruction.

To get around the NP-hardness of sparse approximation, we need more structure in the matrix $\Phi$.
Instead of considering linear independence of all subcollections of $2K$ columns, it has become common to impose a much stronger requirement: that every submatrix of $2K$ columns of $\Phi$ be well-conditioned.
To be explicit, we have the following definition:

\begin{defn}
The matrix $\Phi$ has the \emph{$(K,\delta)$-restricted isometry property (RIP)} if 
\begin{equation*}
(1-\delta)\|x\|^2\leq\|\Phi x\|^2\leq(1+\delta)\|x\|^2
\end{equation*}
for every $K$-sparse vector $x$.
The smallest $\delta$ for which $\Phi$ is $(K,\delta)$-RIP is the \emph{restricted isometry constant (RIC)} $\delta_K$.
\end{defn}

In words, matrices which satisfy RIP act as a near-isometry on sufficiently sparse vectors.
Note that a $(2K,\delta)$-RIP matrix with $\delta<1$ necessarily has that all subcollections of $2K$ columns are linearly independent.
Also, the well-conditioning requirement of RIP forces dissimilarity in the columns of $\Phi$ to provide stability in reconstruction.
Most importantly, the additional structure of RIP allows for the possibility of getting around the NP-hardness of sparse approximation.
Indeed, a significant result in compressed sensing is that RIP sensing matrices enable efficient reconstruction:

\begin{thm}[Theorem 1.3 in \cite{Candes:08}]
\label{thm.rip use}
Suppose an $M\times N$ matrix $\Phi$ has the $(2K,\delta)$-restricted isometry property for some $\delta<\sqrt{2}-1$.
Assuming $\|z\|\leq\varepsilon$, then for every $K$-sparse vector $x\in\mathbb{R}^N$, the following reconstruction from \eqref{eq.cs eqn}:
\begin{equation*}
\tilde{x}=\arg\min\|\hat{x}\|_1\qquad\mbox{s.t. }\|y-\Phi\hat{x}\|\leq\varepsilon
\end{equation*}
satisfies $\|\tilde{x}-x\|\leq C\varepsilon$, where $C$ only depends on $\delta$.
\end{thm}

The fact that RIP sensing matrices convert an NP-hard reconstruction problem into an $\ell_1$-minimization problem has prompted many in the community to construct RIP matrices.
Among these constructions, the most successful have been random matrices, such as matrices with independent Gaussian or Bernoulli entries~\cite{BaraniukDDW:08}, or matrices whose rows were randomly selected from the discrete Fourier transform matrix~\cite{RudelsonV:07}.
With high probability, these random constructions support sparsity levels $K$ on the order of $\smash{\frac{M}{\log^\alpha N}}$ for some $\alpha\geq1$.
Intuitively, this level of sparsity is near-optimal because $K$ cannot exceed $\smash{\frac{M}{2}}$ by the linear independence condition.
Unfortunately, it is difficult to check whether a particular instance of a random matrix is $(K,\delta)$-RIP, as this involves the calculation of singular values for all $\smash{\binom{N}{K}}$ submatrices of $K$ columns of the matrix.
For this reason, and for the sake of reliable sensing standards, many have become interested in finding deterministic RIP matrix constructions.

In the next section, we review the well-understood techniques that are commonly used to analyze the restricted isometry of deterministic constructions: the Gershgorin circle theorem, and the \emph{spark} of a matrix.
Unfortunately, neither technique demonstrates RIP for sparsity levels as large as what random constructions are known to support; rather, with these techniques, a deterministic $M\times N$ matrix $\Phi$ can only be shown to have RIP for sparity levels on the order of $\sqrt{M}$.
This limitation has become known as the ``square-root bottleneck,'' and it poses an important problem in matrix design~\cite{Tao:07}.

To date, the only deterministic construction that manages to go beyond this bottleneck is given by Bourgain et al.~\cite{BourgainDFKK:11}; in Section~3, we discuss what they call \emph{flat RIP}, which is the technique they use to demonstrate RIP.
It is important to stress the significance of their contribution: 
Before~\cite{BourgainDFKK:11}, it was unclear how deterministic analysis might break the bottleneck, and as such, their result is a major theoretical achievement.
On the other hand, their improvement over the square-root bottleneck is notably slight compared to what random matrices provide.
However, by our Theorem~\ref{thm.fro to rip}, their technique can actually be used to demonstrate RIP for sparsity levels much larger than $\sqrt{M}$, meaning one could very well demonstrate random-like performance given the proper construction.
Our result applies their technique to random matrices, and it inadvertently serves as a simple alternative proof that certain random matrices are RIP.
In Section~4, we introduce an alternate technique, which by our Theorem~\ref{thm.power to rip}, can also demonstrate RIP for large sparsity levels.

After considering the efficacy of these techniques to demonstrate RIP, it remains to find a deterministic construction that is amenable to analysis.
To this end, we discuss various properties of a particularly nice matrix which comes from \emph{frame theory}, called an \emph{equiangular tight frame (ETF)}.
Specifically, real ETFs can be characterized in terms of their Gram matrices using strongly regular graphs~\cite{Waldron:09}.
By applying the techniques of Sections~3 and~4 to real ETFs, we derive equivalent combinatorial statements in graph theory.
By focussing on the ETFs which correspond to Paley graphs of prime order, we are able to make important statements about their clique numbers and provide some intuition for an open problem in number theory.
We conclude by conjecturing that the Paley ETFs are RIP in a manner similar to random matrices.

\section{Well-understood techniques}

\subsection{Applying Gershgorin's circle thoerem}

Take an $M\times N$ matrix $\Phi$.
For a given $K$, we wish to find some $\delta$ for which $\Phi$ is $(K,\delta)$-RIP.
To this end, it is useful to consider the following expression for the restricted isometry constant:
\begin{equation}
\label{eq.delta min}
\delta_K=\max_{\substack{\mathcal{K}\subseteq\{1,\ldots,N\}\\|\mathcal{K}|=K}}\|\Phi_\mathcal{K}^*\Phi_\mathcal{K}^{}-\mathrm{I}_K\|_2.
\end{equation}
Here, $\Phi_\mathcal{K}$ denotes the submatrix consisting of columns of $\Phi$ indexed by $\mathcal{K}$.
Note that we are not tasked with actually computing $\delta_K$; rather, we recognize that $\Phi$ is $(K,\delta)$-RIP for every $\delta\geq\delta_K$, and so we seek an upper bound on $\delta_K$.
The following classical result offers a particularly easy-to-calculate bound on eigenvalues:

\begin{thm}[Gershgorin circle theorem~\cite{Gerschgorin:31}]
For each eigenvalue $\lambda$ of a $K\times K$ matrix $A$, there is an index $i\in\{1,\ldots,K\}$ such that
\begin{equation*}
\Big|\lambda-A[i,i]\Big|\leq\sum_{\substack{j=1\\j\neq i}}^K\Big|A[i,j]\Big|.
\end{equation*}
\end{thm}

To use this theorem, take some $\Phi$ with unit-norm columns.
Note that $\Phi_\mathcal{K}^*\Phi_\mathcal{K}^{}$ is the Gram matrix of the columns indexed by $\mathcal{K}$, and as such, the diagonal entries are $1$, and the off-diagonal entries are inner products between distinct columns of $\Phi$.
Let $\mu$ denote the worst-case coherence of $\Phi=[\varphi_1\cdots \varphi_N]$:
\begin{equation*}
\mu:=\max_{\substack{i,j\in\{1,\ldots,N\}\\i\neq j}}|\langle \varphi_i,\varphi_j\rangle|.
\end{equation*}
Then the size of each off-diagonal entry of $\Phi_\mathcal{K}^*\Phi_\mathcal{K}^{}$ is $\leq\mu$, regardless of our choice for $\mathcal{K}$.
Therefore, for every eigenvalue $\lambda$ of $\Phi_\mathcal{K}^*\Phi_\mathcal{K}^{}-\mathrm{I}_K$, the Gershgorin circle theorem gives
\begin{equation}
\label{eq.bound}
|\lambda|
=|\lambda-0|
\leq\sum_{\substack{j=1\\j\neq i}}^K|\langle \varphi_i,\varphi_j\rangle|
\leq(K-1)\mu.
\end{equation}
Since \eqref{eq.bound} holds for every eigenvalue $\lambda$ of $\Phi_\mathcal{K}^*\Phi_\mathcal{K}^{}-\mathrm{I}_K$ and every choice of $\mathcal{K}\subseteq\{1,\ldots,N\}$, we conclude from \eqref{eq.delta min} that $\delta_K\leq(K-1)\mu$, i.e., $\Phi$ is $(K,(K-1)\mu)$-RIP.
This process of using the Gershgorin circle theorem to demonstrate RIP for deterministic constructions has become standard in the community~\cite{ApplebaumHSC:09,DeVore:07,FickusMT:10}.

Recall that random RIP constructions support sparsity levels $K$ on the order of $\smash{\frac{M}{\log^\alpha N}}$ for some $\alpha\geq1$.
To see how well the Gershgorin circle theorem demonstrates RIP, we need to express $\mu$ in terms of $M$ and $N$.
To this end, we consider the following result: 
\begin{thm}[Welch bound~\cite{Welch:74}]
Every $M\times N$ matrix with unit-norm columns has worst-case coherence
\begin{equation*}
\mu\geq\sqrt{\frac{N-M}{M(N-1)}}.
\end{equation*}
\end{thm}
To use this result, we consider matrices whose worst-case coherence achieves equality in the Welch bound.
These are known as equiangular tight frames~\cite{StrohmerH:03}, which can be defined as follows:

\begin{defn}
\label{defn.etf}
A matrix is said to be an \emph{equiangular tight frame (ETF)} if
\begin{itemize}
\item[(i)] the columns have unit norm,
\item[(ii)] the rows are orthogonal with equal norm, and
\item[(iii)] the inner products between distinct columns are equal in modulus. 
\end{itemize}
\end{defn}

To date, there are three general constructions that build several families of ETFs~\cite{FickusMT:10,Waldron:09,XiaZG:05}.
Since ETFs achieve equality in the Welch bound, we can further analyze what it means for an $M\times N$ ETF $\Phi$ to be $(K,(K-1)\mu)$-RIP.
In particular, since Theorem~\ref{thm.rip use} requires that $\Phi$ be $(2K,\delta)$-RIP for $\delta<\sqrt{2}-1$, it suffices to have
$\smash{\frac{2K}{\sqrt{M}}<\sqrt{2}-1}$, since this implies
\begin{equation}
\label{eq.square root bottleneck}
\delta=(2K-1)\mu=(2K-1)\sqrt{\frac{N-M}{M(N-1)}}\leq\frac{2K}{\sqrt{M}}<\sqrt{2}-1.
\end{equation}
That is, ETFs form sensing matrices that support sparsity levels $K$ on the order of $\sqrt{M}$.
Most other deterministic constructions have identical bounds on sparsity levels~\cite{ApplebaumHSC:09,DeVore:07,FickusMT:10}.
In fact, since ETFs minimize coherence, they are necessarily optimal constructions in terms of the Gershgorin demonstration of RIP, but the question remains whether they are actually RIP for larger sparsity levels; the Gershgorin demonstration fails to account for cancellations in the sub-Gram matrices $\Phi_\mathcal{K}^*\Phi_\mathcal{K}^{}$, and so this technique is too weak to indicate either possibility.

\subsection{Spark considerations}

Recall that, in order for an inversion process for \eqref{eq.cs eqn} to exist, $\Phi$ must map $K$-sparse vectors injectively, or equivalently, every subcollection of $2K$ columns of $\Phi$ must be linearly independent.
This linear independence condition can be nicely expressed in more general terms, as the following definition provides:

\begin{defn}
The \emph{spark} of a matrix $\Phi$ is the size of the smallest linearly dependent subset of columns, i.e.,
\begin{equation*}
\mathrm{Spark}(\Phi)
=\min\Big\{\|x\|_0:Fx=0,~x\neq0\Big\}.
\end{equation*}
\end{defn}

This definition was introduced by Dohono and Elad~\cite{DohonoE:03} to help build a theory of sparse representation that later gave birth to modern compressed sensing.
The concept of spark is also found in matroid theory, where it goes by the name \emph{girth}~\cite{AlexeevCM:arxiv11}.
The condition that every subcollection of $2K$ columns of $\Phi$ is linearly independent is equivalent to $\mathrm{Spark}(\Phi)>2K$.
Relating spark to RIP, suppose $\Phi$ is $(K,\delta)$-RIP with $\mathrm{Spark}(\Phi)\leq K$.
Then there exists a nonzero $K$-sparse vector $x$ such that
\begin{equation*}
(1-\delta)\|x\|^2\leq\|\Phi x\|^2=0,
\end{equation*}
and so $\delta\geq1$.
The reason behind this stems from our necessary linear independence condition: RIP implies linear independence, and so small spark implies linear dependence, which in turn implies not RIP.

As an example of using spark to analyze RIP, we now consider a construction that dates back to Seidel~\cite{Seidel:73}, and was recently developed further in~\cite{FickusMT:10}.
Here, a special type of block design is used to build an ETF.
Let's start with a definition:

\begin{defn}
A $(t,k,v)$-\emph{Steiner system} is a $v$-element set $V$ with a collection of $k$-element subsets of $V$, called \emph{blocks}, with the property that any $t$-element subset of $V$ is contained in exactly one block.
The $\{0,1\}$-\emph{incidence matrix} $A$ of a Steiner system has entries $A[i,j]$, where $A[i,j]=1$ if the $i$th block contains the $j$th element, and otherwise $A[i,j]=0$.
\end{defn}

One example of a Steiner system is a set with all possible two-element blocks.
This forms a $(2,2,v)$-Steiner system because every pair of elements is contained in exactly one block.
The following theorem details how to construct ETFs using Steiner systems.

\begin{thm}[Theorem 1 in \cite{FickusMT:10}]
\label{thm.steiner etfs}
Every $(2,k,v)$-Steiner system can be used to build a $\smash{\frac{v(v-1)}{k(k-1)}\times v(1+\frac{v-1}{k-1})}$ equiangular tight frame $\Phi$ according the following procedure:
\begin{itemize}
\item[(i)] Let $A$ be the $\frac{v(v-1)}{k(k-1)}\times v$ incidence matrix of a $(2,k,v)$-Steiner system.
\item[(ii)] Let $H$ be a $(1+\frac{v-1}{k-1})\times(1+\frac{v-1}{k-1})$ (possibly complex) Hadamard matrix.
\item[(iii)] For each $j=1,\ldots,v$, let $\Phi_j$ be a $\frac{v(v-1)}{k(k-1)}\times(1+\frac{v-1}{k-1})$ matrix obtained from the $j$th column of $A$ by replacing each of the one-valued entries with a distinct row of $H$, and every zero-valued entry with a row of zeros.
\item[(iv)] Concatenate and rescale the $\Phi_j$'s to form $\Phi=(\frac{k-1}{v-1})^\frac{1}{2}[\Phi_1\cdots \Phi_v]$.
\end{itemize}
\end{thm}

As an example, we build an ETF from a (2,2,4)-Steiner system.
In this case, we make use of the corresponding incidence matrix $A$ along with a $4\times 4$ Hadamard matrix $H$:
\begin{equation*}
A=\left[\begin{array}{cccc}+&+&&\\+&&+&\\+&&&+\\&+&+&\\&+&&+\\&&+&+\end{array}\right],
\qquad
H=\left[\begin{array}{cccc}+&+&+&+\\+&-&+&-\\+&+&-&-\\+&-&-&+\end{array}\right].
\end{equation*}
In both of these matrices, pluses represent $1$'s, minuses represent $-1$'s, and blank spaces represent $0$'s.
For the matrix $A$, each row represents a block.
Since each block contains two elements, each row of the matrix has two ones.
Also, any two elements determines a unique common row, and so any two columns have a single one in common.
To form the corresponding $6\times 16$ ETF $\Phi$, we replace the three ones in each column of $A$ with the second, third, and fourth rows of $H$.
Normalizing the columns gives the following $6\times 16$ ETF:
\begin{equation}
\label{eq.steiner etf example}
\Phi=\frac{1}{\sqrt{3}}\left[\begin{array}{cccccccccccccccc}
+&-&+&-&+&-&+&-&&&&&&&&\\
+&+&-&-&&&&&+&-&+&-&&&&\\
+&-&-&+&&&&&&&&&+&-&+&-\\
&&&&+&+&-&-&+&+&-&-&&&&\\
&&&&+&-&-&+&&&&&+&+&-&-\\
&&&&&&&&+&-&-&+&+&-&-&+                              
                             \end{array}
\right].
\end{equation}

It is easy to verify that $\Phi$ satisfies Definition~\ref{defn.etf}.
Several infinite families of $(2,k,v)$-Steiner systems are already known, and Theorem~\ref{thm.steiner etfs} says that each one can be used to build a different ETF.
Recall from the previous subsection that Steiner ETFs, being ETFs, are optimal constructions in terms of the Gershgorin demonstration of RIP.
We now use the notion of spark to further analyze Steiner ETFs.
Specifically, note that the first four columns in \eqref{eq.steiner etf example} are linearly dependent.
As such, $\mathrm{Spark}(\Phi)\leq4$.
In general, the spark of a Steiner ETF is $\smash{\leq\frac{v-1}{k-1}\leq\sqrt{2M}}$~(see Theorem~3 of \cite{FickusMT:10} and discussion thereafter), and so having $K$ on the order of $\sqrt{M}$ is \emph{necessary} for a Steiner ETF to be $(K,\delta)$-RIP for some $\delta<1$. 
This answers the closing question of the previous subsection: 
in general, ETFs are not RIP for sparsity levels larger than the order of $\sqrt{M}$.
This contrasts with random constructions, which support sparsity levels as large as the order of $\smash{\frac{M}{\log^\alpha N}}$ for some $\alpha\geq1$.
That said, are there techniques to demonstrate that certain deterministic matrices are RIP for sparsity levels larger than the order of $\sqrt{M}$?

\section{Flat restricted orthogonality}

In~\cite{BourgainDFKK:11}, Bourgain et al.~provided a deterministic construction of $M\times N$ RIP matrices that support sparsity levels $K$ on the order of $M^{1/2+\varepsilon}$ for some small value of $\varepsilon$.
To date, this is the only known deterministic RIP construction that breaks the so-called ``square-root bottleneck.''
In this section, we analyze their technique for demonstrating RIP, but first, we provide some historical context.
We begin with a definition:

\begin{defn}
\label{defn.ro}
The matrix $\Phi$ has \emph{$(K,\theta)$-restricted orthogonality (RO)} if 
\begin{equation*}
|\langle \Phi x, \Phi y\rangle|
\leq\theta\|x\|\|y\|
\end{equation*}
for every pair of $K$-sparse vectors $x,y$ with disjoint support.
The smallest $\theta$ for which $\Phi$ has $(K,\theta)$-RO is the \emph{restricted orthogonality constant (ROC)} $\theta_K$.
\end{defn}

In the past, restricted orthogonality was studied to produce reconstruction performance guarantees for both $\ell_1$-minimization and the Dantzig selector~\cite{CandesT:05,CandesT:07}.
Intuitively, restricted orthogonality is important to compressed sensing because any stable inversion process for \eqref{eq.cs eqn} would require $\Phi$ to map vectors of disjoint support to particularly dissimilar measurements.
For the present paper, we are interested in upper bounds on RICs; in this spirit, the following result illustrates some sort of equivalence between RICs and ROCs:

\begin{lem}[Lemma 1.2 in \cite{CandesT:05}]
$\theta_K
\leq \delta_{2K}
\leq \theta_K+\delta_K$.
\end{lem}

To be fair, the above upper bound on $\delta_{2K}$ does not immediately help in estimating $\delta_{2K}$, as it requires one to estimate $\delta_K$.
Certainly, we may iteratively apply this bound to get
\begin{equation}
\label{eq.current bound}
\delta_{2K}
\leq\theta_K+\theta_{\lceil K/2\rceil}+\theta_{\lceil K/4\rceil}+\cdots+\theta_1+\delta_1
\leq(1+\lceil\log_2 K\rceil)\theta_K+\delta_1.
\end{equation}
Note that $\delta_1$ is particularly easy to calculate:
\begin{equation*}
\delta_1=\max_{n\in\{1,\ldots,N\}}\Big|\|\varphi_n\|^2-1\Big|,
\end{equation*}
which is zero when the columns of $\Phi$ have unit norm.
In pursuit of a better upper bound on~$\delta_{2K}$, we use techniques from \cite{BourgainDFKK:11} to remove the log factor from \eqref{eq.current bound}: 

\begin{lem}
\label{lem.ro to ri}
$\delta_{2K}
\leq 2\theta_K+\delta_1$.
\end{lem}

\begin{proof}
Given a matrix $\Phi=[\varphi_1\cdots\varphi_N]$, we want to upper-bound the smallest $\delta$ for which $(1-\delta)\|x\|^2\leq\|\Phi x\|^2\leq(1+\delta)\|x\|^2$, or equivalently:
\begin{equation}
\label{eq.smallest delta}
\delta\geq\Big|\|\Phi\tfrac{x}{\|x\|}\|^2-1\Big|
\end{equation}
for every nonzero $2K$-sparse vector $x$.
We observe from \eqref{eq.smallest delta} that we may take $x$ to have unit norm without loss of generality.
Letting $\mathcal{K}$ denote a size-$2K$ set that contains the support of $x$, and letting $\{x_k\}_{k\in\mathcal{K}}$ denote the corresponding entries of $x$, the triangle inequality gives
\begin{align}
\nonumber
\Big|\|\Phi x\|^2-1\Big|
&=\bigg|\bigg\langle\sum_{i\in\mathcal{K}}x_i\varphi_i,\sum_{j\in\mathcal{K}}x_j\varphi_j\bigg\rangle-1\bigg|\\
\nonumber
&=\bigg|\sum_{i\in\mathcal{K}}\sum_{\substack{j\in\mathcal{K}\\j\neq i}}\langle x_i\varphi_i,x_j\varphi_j\rangle+\sum_{i\in\mathcal{K}}\|x_i\varphi_i\|^2-1\bigg|\\
\label{eq.smallest delta 2}
&\leq\bigg|\sum_{i\in\mathcal{K}}\sum_{\substack{j\in\mathcal{K}\\j\neq i}}\langle x_i\varphi_i,x_j\varphi_j\rangle\bigg|+\bigg|\sum_{i\in\mathcal{K}}\|x_i\varphi_i\|^2-1\bigg|.
\end{align}
Since $\sum_{i\in\mathcal{K}}|x_i|^2=1$, the second term of \eqref{eq.smallest delta 2} satisfies
\begin{equation}
\label{eq.smallest delta 3}
\bigg|\sum_{i\in\mathcal{K}}\|x_i\varphi_i\|^2-1\bigg|
\leq\sum_{i\in\mathcal{K}}|x_i|^2\Big|\|\varphi_i\|^2-1\Big|
\leq\sum_{i\in\mathcal{K}}|x_i|^2\delta_1
=\delta_1,
\end{equation}
and so it remains to bound the first term of \eqref{eq.smallest delta 2}.
To this end, we note that for each $i,j\in\mathcal{K}$ with $j\neq i$, the term $\langle x_i\varphi_i,x_j\varphi_j\rangle$ appears in
\begin{equation*}
\sum_{\substack{\mathcal{I}\subseteq\mathcal{K}\\|\mathcal{I}|=K}}\sum_{i\in\mathcal{I}}\sum_{j\in\mathcal{K}\setminus\mathcal{I}}\langle x_i\varphi_i,x_j\varphi_j\rangle
\end{equation*}
as many times as there are size-$K$ subsets of $\mathcal{K}$ which contain $i$ but not $j$, i.e., $\binom{2K-2}{K-1}$ times.
Thus, we use the triangle inequality and the definition of restricted orthogonality to get
\begin{align*}
\bigg|\sum_{i\in\mathcal{K}}\sum_{\substack{j\in\mathcal{K}\\j\neq i}}\langle x_i\varphi_i,x_j\varphi_j\rangle\bigg|
&=\bigg|\frac{1}{\binom{2K-2}{K-1}}\sum_{\substack{\mathcal{I}\subseteq\mathcal{K}\\|\mathcal{I}|=K}}\sum_{i\in\mathcal{I}}\sum_{j\in\mathcal{K}\setminus\mathcal{I}}\langle x_i\varphi_i,x_j\varphi_j\rangle\bigg|\\
&\leq\frac{1}{\binom{2K-2}{K-1}}\sum_{\substack{\mathcal{I}\subseteq\mathcal{K}\\|\mathcal{I}|=K}}\bigg|\bigg\langle \sum_{i\in\mathcal{I}}x_i\varphi_i,\sum_{j\in\mathcal{K}\setminus\mathcal{I}}x_j\varphi_j\bigg\rangle\bigg|\\
&\leq\frac{1}{\binom{2K-2}{K-1}}\sum_{\substack{\mathcal{I}\subseteq\mathcal{K}\\|\mathcal{I}|=K}}
\theta_K\bigg(\sum_{i\in\mathcal{I}}|x_i|^2\bigg)^{1/2}\bigg(\sum_{j\in\mathcal{K}\setminus\mathcal{I}}|x_j|^2\bigg)^{1/2}.
\end{align*}
At this point, $x$ having unit norm implies $(\sum_{i\in\mathcal{I}}|x_i|^2)^{1/2}(\sum_{j\in\mathcal{K}\setminus\mathcal{I}}|x_j|^2)^{1/2}\leq\frac{1}{2}$, and so
\begin{equation*}
\bigg|\sum_{i\in\mathcal{K}}\sum_{\substack{j\in\mathcal{K}\\j\neq i}}\langle x_i\varphi_i,x_j\varphi_j\rangle\bigg|
\leq\frac{1}{\binom{2K-2}{K-1}}\sum_{\substack{\mathcal{I}\subseteq\mathcal{K}\\|\mathcal{I}|=K}}
\frac{\theta_K}{2}
=\frac{\binom{2K}{K}}{\binom{2K-2}{K-1}}\frac{\theta_K}{2}
=\bigg(4-\frac{2}{K}\bigg)\frac{\theta_K}{2}.
\end{equation*}
Applying both this and \eqref{eq.smallest delta 3} to \eqref{eq.smallest delta 2} gives the result.
\end{proof}

Having discussed the relationship between restricted isometry and restricted orthogonality, we are now ready to introduce the property used in \cite{BourgainDFKK:11} to demonstrate RIP:

\begin{defn}
The matrix $\Phi=[\varphi_1\cdots\varphi_N]$ has \emph{$(K,\hat\theta)$-flat restricted orthogonality} if 
\begin{equation*}
\bigg|\bigg\langle \sum_{i\in\mathcal{I}}\varphi_i,\sum_{j\in\mathcal{J}}\varphi_j \bigg\rangle\bigg|
\leq\hat\theta(|\mathcal{I}||\mathcal{J}|)^{1/2}
\end{equation*}
for every disjoint pair of subsets $\mathcal{I},\mathcal{J}\subseteq\{1,\ldots,N\}$ with $|\mathcal{I}|,|\mathcal{J}|\leq K$.
\end{defn}

Note that $\Phi$ has $(K,\theta_K)$-flat restricted orthogonality (FRO) by taking $x$ and $y$ in Definition~\ref{defn.ro} to be the characteristic functions $\chi_\mathcal{I}$ and $\chi_\mathcal{J}$, respectively.
Also to be clear, \emph{flat restricted orthogonality} is called \emph{flat RIP} in~\cite{BourgainDFKK:11}; we feel the name change is appropriate considering the preceeding literature.
Moreover, the definition of flat RIP in \cite{BourgainDFKK:11} required $\Phi$ to have unit-norm columns, whereas we strengthen the corresponding results so as to make no such requirement.
Interestingly, FRO bears some resemblence to the cut-norm of the Gram matrix $\Phi^*\Phi$, defined as the maximum value of $|\sum_{i\in\mathcal{I}}\sum_{j\in\mathcal{J}}\langle\varphi_i,\varphi_j\rangle|$ over \emph{all} subsets $\mathcal{I},\mathcal{J}\subseteq\{1,\ldots,N\}$; the cut-norm has received some attention recently for the hardness of its approximation~\cite{AlonN:06}.
The following theorem illustrates the utility of flat restricted orthogonality as an estimate of the RIC:

\begin{thm}
\label{thm.fro}
A matrix with $(K,\hat\theta)$-flat restricted orthogonality has a restricted orthogonality constant $\theta_K$ which is $\leq C\hat\theta\log K$, and we may take $C=75$.
\end{thm}

Indeed, when combined with Lemma~\ref{lem.ro to ri}, this result gives an upper bound on the RIC: $\delta_{2K}\leq 2C\hat\theta\log K + \delta_1$.
The noteworthy benefit of this upper bound is that the problem of estimating singular values of submatrices is reduced to a combinatorial problem of bounding the coherence of disjoint sums of columns.
Furthermore, this reduction comes at the price of a mere log factor in the estimate.
In~\cite{BourgainDFKK:11}, Bourgain et al.~managed to satisfy this combinatorial coherence property using techniques from additive combinatorics.
While we will not discuss their construction, we find the proof of Theorem~\ref{thm.fro} to be instructive; our proof is valid for all values of $K$ (as opposed to sufficiently large $K$ in the original~\cite{BourgainDFKK:11}), and it has near-optimal constants where appropriate.
The proof can be found in the Appendix.

To reiterate, Bourgain et al.~\cite{BourgainDFKK:11} used flat restricted orthogonality to build the only known deterministic construction of $M\times N$ RIP matrices that support sparsity levels $K$ on the order of $M^{1/2+\varepsilon}$ for some small value of $\varepsilon$.
We are particularly interested in the efficacy of FRO as a technique to demonstrate RIP in general.
Certainly, \cite{BourgainDFKK:11} shows that FRO can produce at least an $\varepsilon$ improvement over the Gershgorin technique discussed in the previous section, but it remains to be seen whether FRO can do better.

In the remainder of this section, we will show that flat restricted orthogonality is actually capable of demonstrating RIP with much higher sparsity levels than indicated by~\cite{BourgainDFKK:11}.
Hopefully, this realization will spur further research in deterministic constructions which satisfy FRO.
To evaluate FRO, we investigate how well it performs with random matrices; in doing so, we give an alternative proof that certain random matrices satisfy RIP with high probability:

\begin{thm}
\label{thm.fro to rip}
Construct an $M\times N$ matrix $\Phi$ by drawing each of its entries independently from a Gaussian distribution with mean zero and variance $\frac{1}{M}$, take $C$ to be the constant from Theorem~\ref{thm.fro}, and set $\alpha=0.01$.
Then $\Phi$ has $(K,\frac{(1-\alpha)\delta}{2C\log K})$-flat restricted orthogonality and $\delta_1\leq \alpha\delta$, and therefore the $(2K,\delta)$-restricted isometry property, with high probability provided $M\geq\frac{33C^2}{\delta^2}K\log^2 K\log N$.
\end{thm}

In proving this result, we will make use of the following Bernstein inequality:
\begin{thm}[see \cite{Bernstein:46,Yurinskii:76}]
\label{thm.bernstein}
Let $\{Z_m\}_{m=1}^M$ be independent random variables of mean zero with bounded moments, and suppose there exists $L>0$ such that
\begin{equation}
\label{eq.bernstein requirement}
\mathbb{E}|Z_m|^k
\leq\frac{\mathbb{E}|Z_m|^2}{2}L^{k-2}k!
\end{equation}
for every $k\geq2$.
Then 
\begin{equation}
\label{eq.bernstein conclusion}
\mathrm{Pr}\bigg[\sum_{m=1}^M Z_m\geq2t\bigg(\sum_{m=1}^M\mathbb{E}|Z_m|^2\bigg)^{1/2}\bigg]
\leq e^{-t^2}
\end{equation}
provided $\displaystyle{t\leq\frac{1}{2L}\bigg(\sum_{m=1}^M\mathbb{E}|Z_m|^2\bigg)^{1/2}}$.
\end{thm}

\begin{proof}[Proof of Theorem~\ref{thm.fro to rip}]
Considering Lemma~\ref{lem.ro to ri}, it suffices to show that $\Phi$ has restricted orthogonality and that $\delta_1$ is sufficiently small.
First, to demonstrate restricted orthogonality, it suffices to demonstrate FRO by Theorem~\ref{thm.fro}, and so we will ensure that the following quantity is small:
\begin{equation}
\label{eq.we want small}
\bigg\langle\sum_{i\in\mathcal{I}}\varphi_i,\sum_{j\in\mathcal{J}}\varphi_j\bigg\rangle
=\sum_{m=1}^M\bigg(\sum_{i\in\mathcal{I}}\varphi_i[m]\bigg)\bigg(\sum_{j\in\mathcal{J}}\varphi_j[m]\bigg).
\end{equation}
Notice that $X_m:=\sum_{i\in\mathcal{I}}\varphi_i[m]$ and $Y_m:=\sum_{j\in\mathcal{J}}\varphi_j[m]$ are mutually independent over all $m=1,\ldots,M$ since $\mathcal{I}$ and $\mathcal{J}$ are disjoint.
Also, $X_m$ is Gaussian with mean zero and variance $\frac{|\mathcal{I}|}{M}$, while $Y_m$ similarly has mean zero and variance $\frac{|\mathcal{J}|}{M}$.
Viewed this way, \eqref{eq.we want small} being small corresponds to the sum of independent random variables $Z_m:=X_mY_m$ having its probability measure concentrated at zero.
To this end, Theorem~\ref{thm.bernstein} is naturally applicable, as the absolute central moments of a Gaussian random variable $X$ with mean zero and variance $\sigma^2$ are well known:
\begin{equation*}
\mathbb{E}|X|^k
=\left\{\begin{array}{rl}\sqrt{\frac{2}{\pi}}\sigma^k(k-1)!!&\mbox{ if $k$ odd},\\\sigma^k(k-1)!!&\mbox{ if $k$ even}.\end{array}\right.
\end{equation*}
Since $Z_m=X_mY_m$ is a product of independent Gaussian random variables, this gives
\begin{equation*}
\mathbb{E}|Z_m|^k
=\mathbb{E}|X_m|^k~\mathbb{E}|Y_m|^k
\leq \Big(\frac{|\mathcal{I}|}{M}\Big)^{k/2}\Big(\frac{|\mathcal{J}|}{M}\Big)^{k/2}\Big((k-1)!!\Big)^2
\leq \bigg(\frac{(|\mathcal{I}||\mathcal{J}|)^{1/2}}{M}\bigg)^kk!.
\end{equation*}
Further since $\mathbb{E}|Z_m|^2=\frac{|\mathcal{I}||\mathcal{J}|}{M^2}$, we may define $L:=2\frac{(|\mathcal{I}||\mathcal{J}|)^{1/2}}{M}$ to get \eqref{eq.bernstein requirement}.
Later, we will take $\hat\theta<\delta<\sqrt{2}-1<\frac{1}{2}$.
Considering
\begin{equation*}
t
:=\frac{\hat{\theta}\sqrt{M}}{2}
<\frac{\sqrt{M}}{4}
=\frac{1}{2L}\Big(M\frac{|\mathcal{I}||\mathcal{J}|}{M^2}\Big)^{1/2}
=\frac{1}{2L}\bigg(\sum_{m=1}^M\mathbb{E}|Z_m|^2\bigg)^{1/2},
\end{equation*}
we therefore have \eqref{eq.bernstein conclusion}, which in this case has the form
\begin{equation*}
\mathrm{Pr}\Bigg[\bigg|\bigg\langle\sum_{i\in\mathcal{I}}\varphi_i,\sum_{j\in\mathcal{J}}\varphi_j\bigg\rangle\bigg|\geq\hat{\theta}(|\mathcal{I}||\mathcal{J}|)^{1/2}\Bigg]
\leq 2e^{-M\hat{\theta}^2/4},
\end{equation*}
where the probability is doubled due to the symmetric distribution of $\sum_{m=1}^M Z_m$.
Since we need to account for all possible choices of $\mathcal{I}$ and $\mathcal{J}$, we will perform a union bound.
The total number of choices is given by
\begin{equation*}
\sum_{|\mathcal{I}|=1}^K\sum_{|\mathcal{J}|=1}^K\binom{N}{|\mathcal{I}|}\binom{N-|\mathcal{I}|}{|\mathcal{J}|}
\leq K^2\binom{N}{K}^2
\leq N^{2K},
\end{equation*}
and so the union bound gives
\begin{equation}
\label{eq.probability bound 1}
\mathrm{Pr}\Big[\mbox{$\Phi$ does not have $(K,\hat{\theta})$-FRO}\Big]
\leq 2e^{-M\hat{\theta}^2/4}~N^{2K}
=2\exp\Big(-\frac{M\hat{\theta}^2}{4}+2K\log N\Big).
\end{equation}
Thus, Gaussian matrices tend to have FRO, and hence restricted orthogonality by Theorem~\ref{thm.fro}; this is made more precise below.

Again by Lemma~\ref{lem.ro to ri}, it remains to show that $\delta_1$ is sufficiently small.
To this end, we note that $M\|\varphi_n\|^2$ has chi-squared distribution with $M$ degrees of freedom, and so we can use another (simpler) concentration-of-measure result; see Lemma 1 of~\cite{LaurentM:00}:
\begin{equation*}
\mathrm{Pr}\bigg[\Big|\|\varphi_n\|^2-1\Big|\geq 2\Big(\sqrt{\frac{t}{M}}+\frac{t}{M}\Big)\bigg]\leq 2e^{-t}
\end{equation*}
for any $t>0$.
Specifically, we pick
\begin{equation*}
\delta'
:=2\Big(\sqrt{\frac{t}{M}}+\frac{t}{M}\Big)
\leq\frac{4t}{M},
\end{equation*}
and we perform a union bound over the $N$ choices for $\varphi_n$:
\begin{equation}
\label{eq.probability bound 2}
\mathrm{Pr}\Big[\delta_1>\delta'\Big]
\leq 2\exp\Big(-\frac{M\delta'}{4}+\log N\Big).
\end{equation}
To summarize, Lemma~\ref{lem.ro to ri}, the union bound, Theorem~\ref{thm.fro}, and \eqref{eq.probability bound 1} and \eqref{eq.probability bound 2} give
\begin{align*}
\mathrm{Pr}\Big[\delta_{2K}>\delta\Big]
&\leq\mathrm{Pr}\Big[\theta_K>\frac{(1-\alpha)\delta}{2}\mbox{ or }\delta_1>\alpha\delta\Big]\\
&\leq\mathrm{Pr}\Big[\theta_K>\frac{(1-\alpha)\delta}{2}\Big]+\mathrm{Pr}\Big[\delta_1>\alpha\delta\Big]\\
&\leq\mathrm{Pr}\Big[\mbox{$\Phi$ does not have $\displaystyle{\Big(K,\frac{(1-\alpha)\delta}{2C\log K}\Big)}$-FRO}\Big]+\mathrm{Pr}\Big[\delta_1>\alpha\delta\Big]\\
&\leq2\exp\Big(-\frac{M}{4}\Big(\frac{(1-\alpha)\delta}{2C\log K}\Big)^2+2K\log N\Big)+2\exp\Big(-\frac{M\alpha\delta}{4}+\log N\Big),
\end{align*}
and so $M\geq\frac{33C^2}{\delta^2}K\log^2 K\log N$ gives that $\Phi$ has $(2K,\delta)$-RIP with high probability.
\end{proof}

We note that a version of Theorem~\ref{thm.fro to rip} also holds for matrices whose entries are independent Bernoulli random variables taking values $\pm\frac{1}{\sqrt{M}}$ with equal probability.
In this case, one can again apply Theorem~\ref{thm.bernstein} by comparing moments with those of the Gaussian distribution; also, a union bound with $\delta_1$ will not be necessary since the columns have unit norm, meaning $\delta_1=0$.

\section{Restricted isometry by the power method}

In the previous section, we established the efficacy of flat restricted orthogonality as a technique to demonstrate RIP.
While flat restricted orthogonality has proven useful in the past~\cite{BourgainDFKK:11}, future deterministic RIP constructions might not use this technique.
Indeed, it would be helpful to have other techniques available that demonstrate RIP beyond the square-root bottleneck.
In pursuit of such techniques, we recall that the smallest $\delta$ for which $\Phi$ is $(K,\delta)$-RIP is given in terms of operator norms in \eqref{eq.delta min}.
In addition, we notice that for any self-adjoint matrix $A$,
\[ \|A\|_2=\|\lambda(A)\|_\infty\leq\|\lambda(A)\|_p, \]
where $\lambda(A)$ denotes the spectrum of $A$ with multiplicities.
Let $A=UDU^*$ be the eigenvalue decomposition of~$A$.
When $p$ is even, we can express $\|\lambda(A)\|_p$ in terms of an easy-to-calculate trace:
\[ \|\lambda(A)\|_{p}^{p}=\mathrm{Tr}[D^{p}]=\mathrm{Tr}[(UDU^*)^{p}]=\mathrm{Tr}[A^{p}]. \]
Combining these ideas with the fact that $\|\cdot\|_p\rightarrow\|\cdot\|_\infty$ pointwise leads to the following:

\begin{thm}
\label{thm.power method defn}
Given an $M\times N$ matrix $\Phi$, define
\[ \delta_{K;q}:=\max_{\substack{\mathcal{K}\subseteq\{1,\ldots,N\}\\|\mathcal{K}|=K}}\mathrm{Tr}[(\Phi_\mathcal{K}^*\Phi_\mathcal{K}^{}-\mathrm{I}_K)^{2q}]^\frac{1}{2q}. \]
Then $\Phi$ has the $(K,\delta_{K;q})$-restricted isometry property for every $q\geq1$.  Moreover, the restricted isometry constant of $\Phi$ is approached by these estimates: $\lim_{q\rightarrow\infty}\delta_{K;q}=\delta_K$.
\end{thm}

Similar to flat restricted orthogonality, this \emph{power method} has a combinatorial aspect that prompts one to check every sub-Gram matrix of size $K$; one could argue that the power method is slightly \emph{less} combinatorial, as flat restricted orthogonality is a statement about all pairs of disjoint subsets of size $\leq K$.
Regardless, the work of Bourgain et al.~\cite{BourgainDFKK:11} illustrates that combinatorial properties can be useful, and there may exist constructions to which the power method would be naturally applied.
Moreover, we note that since $\delta_{K;q}$ approaches $\delta_K$, a sufficiently large choice of $q$ should deliver better-than-$\varepsilon$ improvement over the Gershgorin analysis.
How large should $q$ be?
If we assume $\Phi$ has unit-norm columns, taking $q=1$ gives
\begin{equation}
\label{eq.delta 1}
\delta_{K;1}^2
=\max_{\substack{\mathcal{K}\subseteq\{1,\ldots,N\}\\|\mathcal{K}|=K}}\mathrm{Tr}[(\Phi_\mathcal{K}^*\Phi_\mathcal{K}^{}-\mathrm{I}_K)^{2}]
=\max_{\substack{\mathcal{K}\subseteq\{1,\ldots,N\}\\|\mathcal{K}|=K}}
\sum_{i\in\mathcal{K}}\sum_{\substack{j\in\mathcal{K}\\ j\neq i}}|\langle \varphi_i,\varphi_j\rangle|^2
\leq K(K-1)\mu^2,
\end{equation}
where $\mu$ is the worst-case coherence of $\Phi$.
Equality is achieved above whenever $\Phi$ is an ETF, in which case \eqref{eq.delta 1} along with reasoning similar to \eqref{eq.square root bottleneck} demonstrates that $\Phi$ is RIP with sparsity levels on the order of $\sqrt{M}$, as the Gershgorin analysis established.
It remains to be shown how $\delta_{K;2}$ compares.
To make this comparison, we apply the power method to random matrices:

\begin{thm}
\label{thm.power to rip}
Construct an $M\times N$ matrix $\Phi$ by drawing each of its entries independently from a Gaussian distribution with mean zero and variance $\frac{1}{M}$, and take $\delta_{K;q}$ to be as defined in Theorem~\ref{thm.power method defn}.
Then $\delta_{K;q}\leq\delta$, and therefore $\Phi$ has the $(K,\delta)$-restricted isometry property, with high probability provided $M\geq\frac{81}{\delta^2}K^{1+1/q}\log\frac{eN}{K}$.
\end{thm}

While flat restricted orthogonality comes with a negligible penalty of $\log^2K$ in the number of measurements, the power method has a penalty of $K^{1/q}$.
As such, the case $q=1$ uses the order of $K^2$ measurements, which matches our calculation in \eqref{eq.delta 1}.
Moreover, the power method with $q=2$ can demonstrate RIP with $K^{3/2}$ measurements, i.e., $K\sim M^{1/2+1/6}$, which is considerably better than an $\varepsilon$ improvement over the Gershgorin technique.

\begin{proof}[Proof of Theorem~\ref{thm.power to rip}]
Take $t:=\frac{\delta}{3K^{1/2q}}-(\frac{K}{M})^{1/2}$ and pick $\mathcal{K}\subseteq\{1,\ldots,N\}$.
Then Theorem~II.13 of~\cite{DavidsonS:01} states
\begin{equation*}
\mathrm{Pr}\bigg[1-\bigg(\sqrt{\frac{K}{M}}+t\bigg)\leq\sigma_{\min}(\Phi_\mathcal{K})\leq\sigma_{\max}(\Phi_\mathcal{K})\leq1+\bigg(\sqrt{\frac{K}{M}}+t\bigg)\bigg]
\geq 1-2e^{-Mt^2/2}.
\end{equation*}
Continuing, we use the fact that $\lambda(\Phi_\mathcal{K}^*\Phi_\mathcal{K}^{})=\sigma(\Phi_\mathcal{K})^2$ to get
\begin{align}
\nonumber
&1-2e^{-Mt^2/2}\\
\nonumber
&\leq \mathrm{Pr}\bigg[\bigg(1-\bigg(\sqrt{\frac{K}{M}}+t\bigg)\bigg)^2\leq\lambda_{\min}(\Phi_\mathcal{K}^*\Phi_\mathcal{K}^{})\leq\lambda_{\max}(\Phi_\mathcal{K}^*\Phi_\mathcal{K}^{})\leq\bigg(1+\bigg(\sqrt{\frac{K}{M}}+t\bigg)\bigg)^2\bigg]\\
\label{eq.eigenvalue interval}
&\leq \mathrm{Pr}\bigg[1-3\bigg(\sqrt{\frac{K}{M}}+t\bigg)\leq\lambda_{\min}(\Phi_\mathcal{K}^*\Phi_\mathcal{K}^{})\leq\lambda_{\max}(\Phi_\mathcal{K}^*\Phi_\mathcal{K}^{})\leq1+3\bigg(\sqrt{\frac{K}{M}}+t\bigg)\bigg],
\end{align}
where the last inequality follows from the fact that $(\frac{K}{M})^{1/2}+t<1$.
Since $\Phi_\mathcal{K}^*\Phi_\mathcal{K}^{}$ and $\mathrm{I}_K$ are simultaneously diagonalizable, the spectrum of $\Phi_\mathcal{K}^*\Phi_\mathcal{K}^{}-\mathrm{I}_K$ is given by $\lambda(\Phi_\mathcal{K}^*\Phi_\mathcal{K}^{}-\mathrm{I}_K)=\lambda(\Phi_\mathcal{K}^*\Phi_\mathcal{K}^{})-1$.
Combining this with \eqref{eq.eigenvalue interval} then gives
\begin{equation*}
\mathrm{Pr}\bigg[\Big\|\lambda(\Phi_\mathcal{K}^*\Phi_\mathcal{K}^{}-\mathrm{I}_K)\Big\|_\infty\leq 3\bigg(\sqrt{\frac{K}{M}}+t\bigg)\bigg]
\geq 1-2e^{-Mt^2/2}.
\end{equation*}
Considering $\mathrm{Tr}[A^{2q}]^\frac{1}{2q}=\|\lambda(A)\|_{2q}\leq K^\frac{1}{2q}\|\lambda(A)\|_\infty$, we continue:
\begin{equation*}
\mathrm{Pr}\bigg[\mathrm{Tr}[(\Phi_\mathcal{K}^*\Phi_\mathcal{K}^{}-\mathrm{I}_K)^{2q}]^\frac{1}{2q}\leq \delta \bigg]
\geq \mathrm{Pr}\bigg[ K^\frac{1}{2q}\Big\|\lambda(\Phi_\mathcal{K}^*\Phi_\mathcal{K}^{}-\mathrm{I}_K)\Big\|_\infty \leq \delta \bigg]
\geq 1-2e^{-Mt^2/2}.
\end{equation*}
From here, we perform a union bound over all possible choices of $\mathcal{K}$:
\begin{align}
\nonumber
\mathrm{Pr}\bigg[\exists\mathcal{K}\mbox{ s.t. }\mathrm{Tr}[(\Phi_\mathcal{K}^*\Phi_\mathcal{K}^{}-\mathrm{I}_K)^{2q}]^\frac{1}{2q}> \delta\bigg]
&\leq\binom{N}{K}\mathrm{Pr}\bigg[\mathrm{Tr}[(\Phi_\mathcal{K}^*\Phi_\mathcal{K}^{}-\mathrm{I}_K)^{2q}]^\frac{1}{2q}> \delta \bigg]\\
\label{eq.whp 1}
&\leq2\exp\Big(-\frac{Mt^2}{2}+K\log \frac{eN}{K}\Big).
\end{align}
Rearranging $M\geq\frac{81}{\delta^2}K^{1+1/q}\log\frac{eN}{K}$ gives $K^{1/2}\leq\frac{\delta M^{1/2}}{9K^{1/2q}\log^{1/2}(eN/K)}\leq\frac{\delta M^{1/2}}{9K^{1/2q}}$, and so
\begin{equation}
\label{eq.whp 2}
\frac{Mt^2}{2}
=\frac{1}{2}\bigg(\frac{\delta M^{1/2}}{3K^{1/2q}}-K^{1/2}\bigg)^2
\geq\frac{1}{2}\bigg(\frac{2\delta M^{1/2}}{9K^{1/2q}}\bigg)^2
\geq 2K\log\frac{eN}{K}.
\end{equation}
Combining \eqref{eq.whp 1} and \eqref{eq.whp 2} gives the result.
\end{proof}

\section{Equiangular tight frames as RIP candidates}

In Section 2, we observed that equiangular tight frames (ETFs) are optimal RIP matrices under the Gershgorin analysis.
In the present section, we reexamine ETFs as prospective RIP matrices.
Specifically, we consider the possibility that certain classes of $M\times N$ ETFs support sparsity levels $K$ larger than the order of $\sqrt{M}$.
Before analyzing RIP, let's first observe some important features of ETFs.
Recall that Definition~\ref{defn.etf} characterized ETFs in terms of their rows and columns.
Interestingly, \emph{real} ETFs have a natural alternative characterization.

Let $\Phi$ be a real $M\times N$ ETF, and consider the corresponding Gram matrix $\Phi^*\Phi$.
Observing Definition~\ref{defn.etf}, we have from (i) that the diagonal entries of $\Phi^*\Phi$ are 1's.
Also, (iii) indicates that the off-diagonal entries are equal in absolute value (to the Welch bound); since $\Phi$ has real entries, the phase of each off-diagonal entry of $\Phi^*\Phi$ is either positive or negative.
Letting $\mu$ denote the absolute value of the off-diagonal entries, we can decompose the Gram matrix as $\Phi^*\Phi=\mathrm{I}_N+\mu S$, where $S$ is a matrix of zeros on the diagonal and $\pm1$'s on the off-diagonal.
Here, $S$ is referred to as a \emph{Seidel adjacency matrix}, as $S$ encodes the adjacency rule of a simple graph with $i\leftrightarrow j$ whenever $S[i,j]=-1$; this correspondence originated in \cite{vanLintS:66}.

There is an important equivalence class amongst ETFs: given an ETF $\Phi$, one can negate any of the columns to form another ETF $\Phi'$.
Indeed, the ETF properties in Definition~\ref{defn.etf} are easily verified to hold for this new matrix.
For obvious reasons, $\Phi$ and $\Phi'$ are called \emph{flipping equivalent}.
This equivalence plays a key role in the following result, which characterizes real ETFs in terms of a particular class of strongly regular graphs:

\begin{defn}
We say a simple graph $G$ is \emph{strongly regular} of the form $\mathrm{srg}(v,k,\lambda,\mu)$ if
\begin{itemize}
\item[(i)] $G$ has $v$ vertices,
\item[(ii)] every vertex has $k$ neighbors (i.e., $G$ is $k$\emph{-regular}),
\item[(iii)] every two adjacent vertices have $\lambda$ common neighbors, and
\item[(iv)] every two non-adjacent vertices have $\mu$ common neighbors.
\end{itemize}

\end{defn}

\begin{thm}[Corollary 5.6 in \cite{Waldron:09}]
\label{thm.etf to srg}
Every real $M\times N$ equiangular tight frame with $N>M+1$ is flipping equivalent to a frame whose Seidel adjacency matrix corresponds to the join of a vertex with a strongly regular graph of the form
\begin{equation*}
\mathrm{srg}\bigg(N-1,L,\frac{3L-N}{2},\frac{L}{2}\bigg),
\qquad L:=\frac{N}{2}-1+\bigg(1-\frac{N}{2M}\bigg)\sqrt{\frac{M(N-1)}{N-M}}.
\end{equation*}
Conversely, every such graph corresponds to flipping equivalence classes of equiangular tight frames in the same manner.
\end{thm}

The previous two sections illustrated the main issue with the Gershgorin analysis: it ignores important cancellations in the sub-Gram matrices.
We suspect that such cancellations would be more easily observed in a real ETF, since Theorem~\ref{thm.etf to srg} neatly represents the Gram matrix's off-diagonal oscillations in terms of adjacencies in a strongly regular graph.
The following result gives a taste of how useful this graph representation can be:

\begin{thm}
\label{thm.etf clique}
Take a real equiangular tight frame $\Phi$ with worst-case coherence $\mu$, and let $G$ denote the corresponding strongly regular graph in Theorem~\ref{thm.etf to srg}.
Then the restricted isometry constant of $\Phi$ is given by $\delta_K=(K-1)\mu$ for every $K\leq\omega(G)+1$, where $\omega(G)$ denotes the size of the largest clique in $G$.
\end{thm}

\begin{proof}
The Gershgorin analysis \eqref{eq.bound} gives the bound $\delta_K\leq(K-1)\mu$, and so it suffices to prove $\delta_K\geq(K-1)\mu$.
Since $K\leq\omega(G)+1$, there exists a clique of size $K$ in the join of $G$ with a vertex.
Let $\mathcal{K}$ denote the vertices of this clique, and take $S_\mathcal{K}$ to be the corresponding Seidel adjacency submatrix. 
In this case, $S_\mathcal{K}=\mathrm{I}_K-\mathrm{J}_K$, where $\mathrm{J}_K$ is the $K\times K$ matrix of all~1's.
Observing the decomposition $\Phi_\mathcal{K}^*\Phi_\mathcal{K}^{}=\mathrm{I}_K+\mu S_\mathcal{K}$, it follows from \eqref{eq.delta min} that
\begin{equation*}
\delta_K
\geq\|\Phi_\mathcal{K}^*\Phi_\mathcal{K}^{}-\mathrm{I}_K\|_2
=\|\mu S_\mathcal{K}\|_2
=\mu\|\mathrm{I}_K-\mathrm{J}_K\|_2
=(K-1)\mu,
\end{equation*}
which concludes the proof.
\end{proof}

This result indicates that the Gershgoin analysis is tight for all real ETFs, at least for sufficiently small values of $K$.
In particular, in order for a real ETF to be RIP beyond the square-root bottleneck, its graph must have a small clique number.
As an example, note that the first four columns of the Steiner ETF in \eqref{eq.steiner etf example} have negative inner products with each other, and thus the corresponding subgraph is a clique.
In general, each block of an $M\times N$ Steiner ETF, whose size is guaranteed to be $\mathrm{O}(\sqrt{M})$, is a lower-dimensional simplex and therefore has this property; this is an alternative proof that the Gershgorin analysis of Steiner ETFs is tight for $K=\mathrm{O}(\sqrt{M})$.

\subsection{Equiangular tight frames with flat restricted orthogonality}

To find ETFs that are RIP beyond the square-root bottleneck, we must apply better techniques than Gershgorin.
We first consider what it means for an ETF to have $(K,\hat\theta)$-flat restricted orthogonality.
Take a real ETF $\Phi=[\varphi_1\cdots\varphi_N]$ with worst-case coherence $\mu$, and note that the corresponding Seidel adjacency matrix $S$ can be expressed in terms of the usual $\{0,1\}$-adjacency matrix~$A$ of the same graph: $S[i,j]=1-2A[i,j]$ whenever $i\neq j$.
Therefore, for every disjoint $\mathcal{I},\mathcal{J}\subseteq\{1,\ldots,N\}$ with $|\mathcal{I}|,|\mathcal{J}|\leq K$, we want
\begin{align}
\nonumber
\hat\theta(|\mathcal{I}||\mathcal{J}|)^{1/2}
&\geq \bigg|\bigg\langle\sum_{i\in\mathcal{I}}\varphi_i,\sum_{j\in\mathcal{J}}\varphi_j\bigg\rangle\bigg|
= \bigg|\sum_{i\in\mathcal{I}}\sum_{j\in\mathcal{J}}\mu S[i,j]\bigg|\\
\label{eq.edge distribution}
&\qquad = \mu\bigg||\mathcal{I}||\mathcal{J}|-2\sum_{i\in\mathcal{I}}\sum_{j\in\mathcal{J}}A[i,j]\bigg|
= 2\mu\bigg|E(\mathcal{I},\mathcal{J})-\frac{1}{2}|\mathcal{I}||\mathcal{J}|\bigg|,
\end{align}
where $E(\mathcal{I},\mathcal{J})$ denotes the number of edges between $\mathcal{I}$ and $\mathcal{J}$ in the graph.
This condition bears a striking resemblence to the following well-known result in graph theory:

\begin{lem}[Expander mixing lemma \cite{HooryLW:06}]
Given a $d$-regular graph of $n$ vertices, the second largest eigenvalue $\lambda$ of its adjacency matrix satisfies
\begin{equation*}
\bigg|E(\mathcal{I},\mathcal{J})-\frac{d}{n}|\mathcal{I}||\mathcal{J}|\bigg|
\leq \lambda(|\mathcal{I}||\mathcal{J}|)^{1/2}
\end{equation*}
for every pair of vertex subsets $\mathcal{I}, \mathcal{J}$.
\end{lem}

In words, the expander mixing lemma says that the number of edges between vertex subsets of a regular graph is roughly what you would expect in a \emph{random} regular graph.
For this lemma to be applicable to \eqref{eq.edge distribution}, we need the strongly regular graph of Theorem~\ref{thm.etf to srg} to satisfy $\frac{L}{N-1}=\frac{d}{n}\approx\frac{1}{2}$.
Using the formula for $L$, it is not difficult to show that $|\frac{L}{N-1}-\frac{1}{2}|=\mathrm{O}(M^{-1/2})$ provided $N=\mathrm{O}(M)$ and $N\geq 2M$.
Furthermore, the second largest eigenvalue of the strongly regular graph will be $\lambda\approx\frac{1}{2}N^{1/2}$, and so the expander mixing lemma says the optimal $\hat\theta$ is $\leq 2\mu\lambda\approx(\frac{N-M}{M})^{1/2}$ since $\mu=(\frac{N-M}{M(N-1)})^{1/2}$.
This is a rather weak estimate for $\hat\theta$ because the expander mixing lemma does not account for the sizes of $\mathcal{I}$ and $\mathcal{J}$ being $\leq K$.
Put in this light, a real ETF that has flat restricted orthogonality corresponds to a strongly regular graph that satisfies a particularly strong version of the expander mixing lemma.

\subsection{Equiangular tight frames and the power method}
Next, we try applying the power method to ETFs.
Given a real ETF $\Phi=[\varphi_1\cdots\varphi_N]$, let $H:=\Phi^*\Phi-\mathrm{I}_N$ denote the ``hollow'' Gram matrix.
Also, take $E_\mathcal{K}$ to be the $N\times K$ matrix built from the columns of $\mathrm{I}_N$ that are indexed by $\mathcal{K}$.
Then
\begin{equation*}
\mathrm{Tr}[(\Phi_\mathcal{K}^*\Phi_\mathcal{K}^{}-\mathrm{I}_K)^{2q}]
=\mathrm{Tr}[(E_\mathcal{K}^*\Phi^*\Phi E_\mathcal{K}^{}-\mathrm{I}_K)^{2q}]
=\mathrm{Tr}[(E_\mathcal{K}^*H E_\mathcal{K}^{})^{2q}]
=\mathrm{Tr}[(H E_\mathcal{K}^{}E_\mathcal{K}^*)^{2q}].
\end{equation*}
Since $E_\mathcal{K}^{}E_\mathcal{K}^*=\sum_{k\in\mathcal{K}}\delta_k^{}\delta_k^*$, where $\delta_k$ is the $k$th identity basis element, we continue:
\begin{align}
\nonumber
\mathrm{Tr}[(\Phi_\mathcal{K}^*\Phi_\mathcal{K}^{}-\mathrm{I}_K)^{2q}]
&=\mathrm{Tr}\bigg[\bigg(H \sum_{k\in\mathcal{K}}\delta_k^{}\delta_k^*\bigg)^{2q}\bigg]\\
\nonumber
&=\sum_{k_0\in\mathcal{K}}\cdots\sum_{k_{2q-1}\in\mathcal{K}}\mathrm{Tr}[H \delta_{k_0}^{}\delta_{k_0}^*\cdots H \delta_{k_{2q-1}}^{}\delta_{k_{2q-1}}^*]\\
\label{eq.zero terms}
&=\sum_{k_0\in\mathcal{K}}\cdots\sum_{k_{2q-1}\in\mathcal{K}}\delta_{k_0}^*H \delta_{k_{1}}^{}\cdots \delta_{k_{2q-1}}^*H \delta_{k_0}^{},
\end{align}
where the last step used the cyclic property of the trace.
From here, note that $H$ has a zero diagonal, meaning several of the terms in \eqref{eq.zero terms} are zero, namely, those for which $k_{\ell+1}=k_\ell$ for some $\ell\in\mathbb{Z}_{2q}$.
To simplify \eqref{eq.zero terms}, take $\mathcal{K}^{(2q)}$ to be the set of $2q$-tuples satisfying $k_{\ell+1}\neq k_\ell$ for every $\ell\in\mathbb{Z}_{2q}$:
\begin{equation}
\label{eq.power method combinatorics}
\mathrm{Tr}[(\Phi_\mathcal{K}^*\Phi_\mathcal{K}^{}-\mathrm{I}_K)^{2q}]
=\sum_{\{k_\ell\}\in\mathcal{K}^{(2q)}} \prod_{\ell\in\mathbb{Z}_{2q}}\langle\varphi_{k_\ell},\varphi_{k_{\ell+1}}\rangle=\mu^{2q}\sum_{\{k_\ell\}\in\mathcal{K}^{(2q)}} \prod_{\ell\in\mathbb{Z}_{2q}}S[k_{\ell},k_{\ell+1}],
\end{equation}
where $\mu$ is the wost-case coherence of $\Phi$, and $S$ is the corresponding Seidel adjacency matrix.
Note that the left-hand side is necessarily nonnegative, while it is not immediate why the right-hand side should be.
This indicates that more simplification can be done, but for the sake of clarity, we will perform this simplification in the special case where $q=2$; the general case is very similar.
When $q=2$, we are concerned with 4-tuples $\{k_0,k_1,k_2,k_3\}\in\mathcal{K}^{(4)}$.
Let's partition these 4-tuples according to the value taken by $k_0$ and $k_q=k_2$.
Note, for a fixed $k_0$ and $k_2$, that $k_1$ can be any value other than $k_0$ or $k_2$, as can $k_3$.
This leads to the following simplification:
\begin{align*}
\sum_{\{k_\ell\}\in\mathcal{K}^{(4)}} \prod_{\ell\in\mathbb{Z}_{4}}S[k_{\ell},k_{\ell+1}]
&=\sum_{k_0\in\mathcal{K}}\sum_{k_2\in\mathcal{K}}\bigg(\sum_{\substack{k_1\in\mathcal{K}\\k_0\neq k_1\neq k_2}}S[k_0,k_1]S[k_1,k_2]\bigg)\bigg(\sum_{\substack{k_3\in\mathcal{K}\\k_2\neq k_3\neq k_0}}S[k_2,k_3]S[k_3,k_0]\bigg)\\
&=\sum_{k_0\in\mathcal{K}}\sum_{k_2\in\mathcal{K}}~~~~\bigg|\!\!\!\!\sum_{\substack{k\in\mathcal{K}\\k_0\neq k\neq k_2}}S[k_0,k]S[k,k_2]\bigg|^2\\
&=\sum_{k_0\in\mathcal{K}}\bigg|\sum_{\substack{k\in\mathcal{K}\\k\neq k_0}}S[k_0,k]S[k,k_0]\bigg|^2+\sum_{k_0\in\mathcal{K}}\sum_{\substack{k_2\in\mathcal{K}\\k_2\neq k_0}}~~~~\bigg|\!\!\!\!\sum_{\substack{k\in\mathcal{K}\\k_0\neq k\neq k_2}}S[k_0,k]S[k,k_2]\bigg|^2.
\end{align*}
The first term above is $K(K-1)^2$, while the other term is not as easy to analyze, as we expect a certain degree of cancellation.
Substituting this simplification into \eqref{eq.power method combinatorics} gives
\begin{equation*}
\mathrm{Tr}[(\Phi_\mathcal{K}^*\Phi_\mathcal{K}^{}-\mathrm{I}_K)^4]
=\mu^4\bigg(K(K-1)^2+\sum_{k_0\in\mathcal{K}}\sum_{\substack{k_2\in\mathcal{K}\\k_2\neq k_1}}~~~~\bigg|\!\!\!\!\sum_{\substack{k\in\mathcal{K}\\k_0\neq k\neq k_2}}S[k_0,k]S[k,k_2]\bigg|^2\bigg).
\end{equation*}
If there were no cancellations in the second term, then it would equal $K(K-1)(K-2)^2$, thereby dominating the expression.
However, if oscillations occured as a $\pm1$ Bernoulli random variable, we could expect this term to be on the order of $K^3$, matching the order of the first term.
In this hypothetical case, since $\mu\leq M^{-1/2}$, the parameter $\delta_{K;2}^4$ defined in Theorem~\ref{thm.power method defn} scales as
$\frac{K^3}{M^2}$, and so $M\sim K^{3/2}$; this corresponds to the behavior exhibited in Theorem~\ref{thm.power to rip}.
To summarize, much like flat restricted orthogonality, applying the power method to ETFs leads to interesting combinatorial questions regarding subgraphs, even when $q=2$.

\subsection{The Paley equiangular tight frame as an RIP candidate}

Pick some prime $p\equiv 1\bmod 4$, and build an $M\times p$ matrix $H$ by selecting the $M:=\frac{p+1}{2}$ rows of the $p\times p$ discrete Fourier transform matrix which are indexed by $Q$, the quadratic residues modulo~$p$ (including zero).
To be clear, the entries of $H$ are scaled to have unit modulus.
Next, take $D$ to be an $M\times M$ diagonal matrix whose zeroth diagonal entry is $p^{-1/2}$, and whose remaining $M-1$ entries are $(\frac{2}{p})^{1/2}$.
Now build the matrix $\Phi$ by concatenating $DH$ with the zeroth identity basis element; for example, when $p=5$, we have a $3\times 6$ matrix:
\begin{equation*}
\Phi
=\left[\begin{array}{llllll} 
\sqrt{\frac{1}{5}}\quad\quad\quad&\sqrt{\frac{1}{5}}\quad\quad\quad&\sqrt{\frac{1}{5}}&\sqrt{\frac{1}{5}}\quad\quad\quad&\sqrt{\frac{1}{5}}\quad\quad\quad&1\\
\sqrt{\frac{2}{5}}&\sqrt{\frac{2}{5}}e^{-2\pi\mathrm{i}/5}&\sqrt{\frac{2}{5}}e^{-2\pi\mathrm{i}2/5}&\sqrt{\frac{2}{5}}e^{-2\pi\mathrm{i}3/5}&\sqrt{\frac{2}{5}}e^{-2\pi\mathrm{i}4/5}&0\\
\sqrt{\frac{2}{5}}&\sqrt{\frac{2}{5}}e^{-2\pi\mathrm{i}4/5}&\sqrt{\frac{2}{5}}e^{-2\pi\mathrm{i}3/5}&\sqrt{\frac{2}{5}}e^{-2\pi\mathrm{i}2/5}&\sqrt{\frac{2}{5}}e^{-2\pi\mathrm{i}/5}&0\\     
       \end{array}\right].
\end{equation*}
We claim that in general, this process produces an $M\times 2M$ equiangular tight frame, which we call the \emph{Paley ETF}~\cite{Renes:07}.
Presuming for the moment that this claim is true, we have the following result which lends hope for the Paley ETF as an RIP matrix:

\begin{lem}
\label{lem.paley ric}
An $M\times 2M$ Paley equiangular tight frame has restricted isometry constant $\delta_K<1$ for all $K\leq M$.
\end{lem}

\begin{proof}
First, we note that Theorem~6 of \cite{AlexeevCM:arxiv11} used Chebotar\"{e}v's theorem \cite{StevenhagenL:96} to prove that the spark of the $M\times 2M$ Paley ETF $\Phi$ is $M+1$, that is, every size-$M$ subcollection of columns of $\Phi$ forms a spanning set.
Thus, for every $\mathcal{K}\subseteq\{1,\ldots,2M\}$ of size $\leq M$, the smallest singular value of $\Phi_\mathcal{K}$ is positive.
It remains to show that the square of the largest singular value is strictly less than 2.
Let $x$ be a unit vector for which $\|\Phi_\mathcal{K}^*x\|=\|\Phi_\mathcal{K}^*\|_2$.
Then since the spark of $\Phi$ is $M+1$, the columns of $\Phi_{\mathcal{K}^\mathrm{c}}$ span, and so
\begin{equation*}
\|\Phi_\mathcal{K}\|_2^2
=\|\Phi_\mathcal{K}^*\|_2^2
=\|\Phi_\mathcal{K}^*x\|^2
<\|\Phi_\mathcal{K}^*x\|^2+\|\Phi_{\mathcal{K}^\mathrm{c}}^*x\|^2
=\|\Phi^* x\|^2
\leq\|\Phi^*\|_2^2
=\|\Phi\Phi^*\|_2
= 2,
\end{equation*}
where the final step follows from Definition~\ref{defn.etf}(i)-(ii), which imply $\Phi\Phi^*=2\mathrm{I}_M$.
\end{proof}

Now that we have an interest in the Paley ETF $\Phi$, we wish to verify that it is, in fact, an ETF.
It suffices to show that the columns of $\Phi$ have unit norm, and that the inner products between distinct columns equal the Welch bound in absolute value.
Certainly, the zeroth identity basis element is unit-norm, while the squared norm of each of the other columns is given by $\frac{1}{p}+(M-1)\frac{2}{p}=\frac{2M-1}{p}=1$.
Also, the inner product between the zeroth identity basis element and any other column equals the zeroth entry of that column: $p^{-1/2}=
(\frac{N-M}{M(N-1)})^{1/2}$.
It remains to calculate the inner product between distinct columns which are not identity basis elements.
To this end, note that since $a^2=b^2$ if and only if $a=\pm b$, the sequence $\{k^2\}_{k=1}^{p-1}\subseteq\mathbb{Z}_p$ doubly covers $Q\setminus\{0\}$, and so
\begin{equation*}
\langle\varphi_n,\varphi_{n'}\rangle
=\frac{1}{p}+\sum_{m\in Q\setminus\{0\}}\bigg(\sqrt{\frac{2}{p}}e^{-2\pi\mathrm{i}mn/p}\bigg)\bigg(\sqrt{\frac{2}{p}}e^{2\pi\mathrm{i}mn'/p}\bigg)
=\frac{1}{p}\sum_{k=0}^{p-1}e^{2\pi\mathrm{i}(n'-n)k^2/p}.
\end{equation*}
This well-known expression is called a quadratic Gauss sum, and since $p\equiv 1\bmod 4$, its value is determined by the Legendre symbol in the following way: $\langle\varphi_n,\varphi_{n'}\rangle=\frac{1}{\sqrt{p}}(\frac{n'-n}{p})$ for every $n,n'\in\mathbb{Z}_p$ with $n\neq n'$, where
\begin{equation*}
\bigg(\frac{k}{p}\bigg)
:=\left\{\begin{array}{rl}+1&\mbox{ if $k$ is a nonzero quadratic residue modulo $p$,}\\0&\mbox{ if $k=0$,}\\-1&\mbox{ otherwise.} \end{array}\right.
\end{equation*}

Having established that $\Phi$ is an ETF, we notice that the inner products between distinct columns of $\Phi$ are real.
This implies that the columns of $\Phi$ can be unitarily rotated to form a real ETF $\Psi$; indeed, one may take $\Psi$ to be the $M\times2M$ matrix formed by taking the nonzero rows of $L^\mathrm{T}$ in the Cholesky factorization $\Phi^*\Phi=LL^\mathrm{T}$.
As such, we consider the Paley ETF to be real.
From here, Theorem~\ref{thm.etf to srg} prompts us to find the corresponding strongly regular graph.
First, we can flip the identity basis element so that its inner products with the other columns of $\Phi$ are all negative.
As such, the corresponding vertex in the graph will be adjacent to each of the other vertices; naturally, this will be the vertex to which the strongly regular graph is joined.
For the remaining vertices, $n\leftrightarrow n'$ precisely when $(\frac{n'-n}{p})=-1$, that is, when $n'-n$ is not a quadratic residue.
The corresponding subgraph is therefore the complement of the Paley graph, namely, the Paley graph \cite{Sachs:62}.
In general, Paley graphs of order $p$ necessarily have $p\equiv 1\bmod 4$, and so this correspondence is particularly natural.

One interesting thing about the Paley ETF's restricted isometry is that it lends insight into important properties of the Paley graph.
The following is the best known upper bound for the clique number of the Paley graph of prime order (see Theorem 13.14 of \cite{Bollobas:01} and discussion thereafter), and we give a new proof of this bound using restricted isometry:

\begin{thm}
\label{thm.clique upper bound}
Let $G$ denote the Paley graph of prime order $p$.
Then the size of the largest clique is $\omega(G)<\sqrt{p}$. 
\end{thm}

\begin{proof}
We start by showing $\omega(G)+1\leq M$.
Suppose otherwise: that there exists a clique $\mathcal{K}$ of size $M+1$ in the join of a vertex with $G$.
Then the corresponding sub-Gram matrix of the Paley ETF has the form $\Phi_\mathcal{K}^*\Phi_\mathcal{K}^{}=(1+\mu)\mathrm{I}_{M+1}-\mu\mathrm{J}_{M+1}$, where $\mu=p^{-1/2}$ is the worst-case coherence and $\mathrm{J}_{M+1}$ is the $(M+1)\times (M+1)$ matrix of 1's.
Since the largest eigenvalue of $\mathrm{J}_{M+1}$ is $M+1$, the smallest eigenvalue of $\Phi_\mathcal{K}^*\Phi_\mathcal{K}^{}$ is $1+p^{-1/2}-(M+1)p^{-1/2}=1-\frac{1}{2}(p+1)p^{-1/2}$, which is negative when $p\geq 5$, contradicting the fact that $\Phi_\mathcal{K}^*\Phi_\mathcal{K}^{}$ is positive semidefinite.

Since $\omega(G)+1\leq M$, we can apply Lemma~\ref{lem.paley ric} and Theorem~\ref{thm.etf clique} to get
\begin{equation}
\label{eq.clique vs 1}
1>\delta_{\omega(G)+1}=\Big(\omega(G)+1-1\Big)\mu=\frac{\omega(G)}{\sqrt{p}},
\end{equation}
and rearranging gives the result.
\end{proof}

It is common to apply probabilistic and heuristic reasoning to gain intuition in number theory. 
For example, consecutive entries of the Legendre symbol are known to mimic certain properties of a $\pm1$ Bernoulli random variable \cite{Peralta:92}.
Moreover, Paley graphs enjoy a certain quasi-random property that was studied in \cite{ChungGW:89}.
On the other hand, Graham and Ringrose~\cite{GrahamR:90} showed that, while random graphs of size $p$ have an expected clique number of $(1+o(1))2\log p/\log 2$, Paley graphs of prime order deviate from this random behavior, having a clique number $\geq c\log p\log\log\log p$ infinitely often.
The best known universal lower bound, $(1/2+o(1))\log p/\log 2$, is given in \cite{Cohen:88}, which indicates that the random graph analysis is at least tight in some sense.
Regardless, this has a significant difference from the upper bound $\sqrt{p}$ in Theorem~\ref{thm.clique upper bound}, and it would be nice if probabilistic arguments could be leveraged to improve this bound, or at least provide some intuition.

Note that our proof \eqref{eq.clique vs 1} hinged on the fact that $\delta_{\omega(G)+1}<1$, courtesy of Lemma~\ref{lem.paley ric}.
Hence, any improvement to our estimate for $\delta_{\omega(G)+1}$ would directly lead to the best known upper bound on the Paley graph's clique number.
To approach such an improvement, note that for large $p$, the Fourier portion of the Paley ETF $DH$ is not significatly different from the normalized partial Fourier matrix $(\frac{2}{p+1})^{1/2}H$; indeed, $\|H_\mathcal{K}^*D^2H_\mathcal{K}^{}-\frac{2}{p+1}H_\mathcal{K}^*H_\mathcal{K}^{}\|_2\leq\frac{2}{p}$ for every $\mathcal{K}\subseteq\mathbb{Z}_p$ of size $\leq\frac{p+1}{2}$, and so the difference vanishes.
If we view the quadratic residues modulo $p$ (the row indices of $H$) as random, then a random partial Fourier matrix serves as a proxy for the Fourier portion of the Paley ETF.
This in mind, we appeal to the following:

\begin{thm}[Theorem 3.2 in \cite{Rauhut:08}]
Draw rows from the $N\times N$ discrete Fourier transform matrix uniformly at random with replacement to construct an $M\times N$ matrix, and then normalize the columns to form $\Phi$.
Then $\Phi$ has restricted isometry constant $\delta_K\leq\delta$ with probability $1-\varepsilon$ provided $\frac{M}{\log M}\geq \frac{C}{\delta^2}K\log^2 K\log N\log\varepsilon^{-1}$, where $C$ is a universal constant.
\end{thm}

In our case, both $M$ and $N$ scale as $p$, and so picking $\delta$ to achieve equality above gives
\begin{equation*}
\delta^2
=\frac{C'}{p}K\log^2K\log^2 p\log\varepsilon^{-1}.
\end{equation*}
Continuing as in \eqref{eq.clique vs 1}, denote $\omega=\omega(G)$ and take $K=\omega$ to get
\begin{equation*}
\frac{C'}{p}\omega\log^2\omega\log^2 p\log\varepsilon^{-1}
\geq\delta_{\omega}^2
=\frac{(\omega-1)^2}{p}
\geq\frac{\omega^2}{2p},
\end{equation*}
and then rearranging gives  $\omega/\log^2\omega\leq C''\log^2p\log\varepsilon^{-1}$ with probability $1-\varepsilon$.
Interestingly, having $\omega/\log^2\omega=\mathrm{O}(\log^3p)$ with high probability (again, under the model that quadratic residues are random) agrees with the results of Graham and Ringrose~\cite{GrahamR:90}.
This gives some intuition for what we can expect the size of the Paley graph's clique number to be, while at the same time demonstrating the power of Paley ETFs as RIP candidates. 
We conclude with the following, which can be reformulated in terms of both flat restricted orthogonality and the power method:

\begin{conj}
The Paley equiangular tight frame has the $(K,\delta)$-restricted isometry property with some $\delta<\sqrt{2}-1$ whenever $K\leq\frac{Cp}{\log^\alpha p}$, for some universal constants $C$ and $\alpha$.
\end{conj}

\section{Appendix}

In this section, we prove Theorem~\ref{thm.fro}, which states that a matrix with $(K,\hat\theta)$-flat restricted orthogonality has $\theta_K\leq C\hat\theta\log K$, that is, it has restricted orthogonality.
The proof below is adapted from the proof of Lemma 3 in~\cite{BourgainDFKK:11}. 
Our proof has the benefit of being valid for all values of $K$ (as opposed to sufficiently large $K$ in the original~\cite{BourgainDFKK:11}), and it has near-optimal constants where appropriate.
Moreover in this version, the columns of the matrix are not required to have unit norm.

\begin{proof}[Proof of Theorem~\ref{thm.fro}]
Given arbitrary disjoint subsets $\mathcal{I},\mathcal{J}\subseteq\{1,\ldots,N\}$ with $|\mathcal{I}|,|\mathcal{J}|\leq K$, we will bound the following quantity three times, each time with different constraints on $\{x_i\}_{i\in\mathcal{I}}$ and $\{y_j\}_{j\in\mathcal{J}}$:
\begin{equation}
\label{eq.to bound}
\bigg|\bigg\langle \sum_{i\in\mathcal{I}}x_i\varphi_i,\sum_{j\in\mathcal{J}}y_j\varphi_j \bigg\rangle\bigg|.
\end{equation}
To be clear, our third bound will have no constraints on $\{x_i\}_{i\in\mathcal{I}}$ and $\{y_j\}_{j\in\mathcal{J}}$, thereby demonstrating restricted orthogonality.
Note that by assumption, \eqref{eq.to bound} is $\leq\hat\theta(|\mathcal{I}||\mathcal{J}|)^{1/2}$ whenever the $x_i$'s and $y_j$'s are in $\{0,1\}$.
We first show that this bound is preserved when we relax the $x_i$'s and $y_j$'s to lie in the interval $[0,1]$.

Pick a disjoint pair of subsets $\mathcal{I}',\mathcal{J}'\subseteq\{1,\ldots,N\}$ with $|\mathcal{I}'|,|\mathcal{J}'|\leq K$.
Starting with some $k\in\mathcal{I}'$, note that flat restricted orthogonality gives that
\begin{align*}
\bigg|\bigg\langle \sum_{i\in\mathcal{I}}\varphi_i,\sum_{j\in\mathcal{J}}\varphi_j \bigg\rangle\bigg|
&\leq\hat\theta(|\mathcal{I}||\mathcal{J}|)^{1/2},\\
\bigg|\bigg\langle \sum_{i\in\mathcal{I}\setminus\{k\}}\varphi_i,\sum_{j\in\mathcal{J}}\varphi_j \bigg\rangle\bigg|
&\leq\hat\theta(|\mathcal{I}\setminus\{k\}||\mathcal{J}|)^{1/2}
\leq\hat\theta(|\mathcal{I}||\mathcal{J}|)^{1/2}
\end{align*}
for every disjoint $\mathcal{I},\mathcal{J}\subseteq\{1,\ldots,N\}$ with $|\mathcal{I}|,|\mathcal{J}|\leq K$ and $k\in\mathcal{I}$.
Thus, we may take any $x_k\in[0,1]$ to form a convex combination of these two expressions, and then the triangle inequality gives
\begin{align}
\nonumber
\hat\theta(|\mathcal{I}||\mathcal{J}|)^{1/2}
&\geq x_k\bigg|\bigg\langle \sum_{i\in\mathcal{I}}\varphi_i,\sum_{j\in\mathcal{J}}\varphi_j \bigg\rangle\bigg|+(1-x_k)\bigg|\bigg\langle \sum_{i\in\mathcal{I}\setminus\{k\}}\varphi_i,\sum_{j\in\mathcal{J}}\varphi_j \bigg\rangle\bigg|\\
\nonumber
&\geq\bigg|x_k\bigg\langle \sum_{i\in\mathcal{I}}\varphi_i,\sum_{j\in\mathcal{J}}\varphi_j \bigg\rangle+(1-x_k)\bigg\langle \sum_{i\in\mathcal{I}\setminus\{k\}}\varphi_i,\sum_{j\in\mathcal{J}}\varphi_j \bigg\rangle\bigg|\\
\label{eq.convex trick}
&=\bigg|\bigg\langle \sum_{i\in\mathcal{I}}\bigg\{\begin{array}{cc}x_k,&i=k\\1,&i\neq k\end{array}\bigg\}\varphi_i,\sum_{j\in\mathcal{J}}\varphi_j \bigg\rangle\bigg|.
\end{align}
Since \eqref{eq.convex trick} holds for every disjoint $\mathcal{I},\mathcal{J}\subseteq\{1,\ldots,N\}$ with $|\mathcal{I}|,|\mathcal{J}|\leq K$ and $k\in\mathcal{I}$, we can do the same thing with an additional index $i\in\mathcal{I}'$ or $j\in\mathcal{J}'$, and replace the corresponding unit coefficient with some $x_i$ or $y_j$ in $[0,1]$.
Continuing in this way proves the claim that \eqref{eq.to bound} is $\leq\hat\theta(|\mathcal{I}||\mathcal{J}|)^{1/2}$ whenever the $x_i$'s and $y_j$'s lie in the interval $[0,1]$.

For the second bound, we assume the $x_i$'s and $y_j$'s are nonnegative with unit norm: $\sum_{i\in\mathcal{I}}x_i^2=\sum_{j\in\mathcal{J}}y_j^2=1$.
To bound \eqref{eq.to bound} in this case, we partition $\mathcal{I}$ and $\mathcal{J}$ according to the size of the corresponding coefficients:
\begin{equation*}
\mathcal{I}_k:=\{i\in\mathcal{I}:2^{-(k+1)}<x_i\leq 2^{-k}\},
\qquad
\mathcal{J}_k:=\{j\in\mathcal{J}:2^{-(k+1)}<y_j\leq 2^{-k}\}.
\end{equation*}
Note the unit-norm constraints ensure that $\mathcal{I}=\bigcup_{k=0}^\infty\mathcal{I}_k$ and $\mathcal{J}=\bigcup_{k=0}^\infty\mathcal{J}_k$.
The triangle inequality thus gives
\begin{align}
\nonumber
\bigg|\bigg\langle \sum_{i\in\mathcal{I}}x_i\varphi_i,\sum_{j\in\mathcal{J}}y_j\varphi_j \bigg\rangle\bigg|
&=\bigg|\bigg\langle \sum_{k_1=0}^\infty\sum_{i\in\mathcal{I}_{k_1}}x_i\varphi_i,\sum_{k_2=0}^\infty\sum_{j\in\mathcal{J}_{k_2}}y_j\varphi_j \bigg\rangle\bigg|\\
\label{eq.partition across sizes}
&\leq\sum_{k_1=0}^\infty\sum_{k_2=0}^\infty 2^{-(k_1+k_2)} \bigg|\bigg\langle \sum_{i\in\mathcal{I}_{k_1}}\frac{x_i}{2^{-k_1}}\varphi_i,\sum_{j\in\mathcal{J}_{k_2}}\frac{y_j}{2^{-k_2}}\varphi_j \bigg\rangle\bigg|.
\end{align}
By the definitions of $\mathcal{I}_{k_1}$ and $\mathcal{J}_{k_2}$, the coefficients of $\varphi_i$ and $\varphi_j$ in \eqref{eq.partition across sizes} all lie in $[0,1]$.
As such, we continue by applying our first bound:
\begin{align}
\nonumber
\bigg|\bigg\langle \sum_{i\in\mathcal{I}}x_i\varphi_i,\sum_{j\in\mathcal{J}}y_j\varphi_j \bigg\rangle\bigg|
&\leq\sum_{k_1=0}^\infty\sum_{k_2=0}^\infty 2^{-(k_1+k_2)} \hat\theta (|\mathcal{I}_{k_1}||\mathcal{J}_{k_2}|)^{1/2}\\
\label{eq.partition across sizes 2}
&=\hat\theta\bigg(\sum_{k=0}^\infty2^{-k}|\mathcal{I}_{k}|^{1/2}\bigg)\bigg(\sum_{k=0}^\infty2^{-k}|\mathcal{J}_{k}|^{1/2}\bigg).
\end{align}
We now observe from the definition of $\mathcal{I}_k$ that
\begin{equation*}
1
=\sum_{i\in\mathcal{I}}x_i^2
=\sum_{k=0}^\infty\sum_{i\in\mathcal{I}_k}x_i^2
>\sum_{k=0}^\infty 4^{-(k+1)}|\mathcal{I}_k|.
\end{equation*}
Thus for any positive integer $t$, the Cauchy-Schwarz inequality gives
\begin{align}
\nonumber
\sum_{k=0}^\infty 2^{-k}|\mathcal{I}_k|^{1/2}
&=\sum_{k=0}^{t-1} 2^{-k}|\mathcal{I}_k|^{1/2}+\sum_{k=t}^\infty 2^{-k}|\mathcal{I}_k|^{1/2}\\
\nonumber
&\leq t^{1/2}\bigg(\sum_{k=0}^{t-1} 4^{-k}|\mathcal{I}_k|\bigg)^{1/2}+\sum_{k=t}^\infty 2^{-k}K^{1/2}\\
\label{eq.partition across sizes 3}
&< 2(t^{1/2}+K^{1/2}2^{-t}),
\end{align}
and similarly for the $\mathcal{J}_k$'s.
For a fixed $K$, we note that \eqref{eq.partition across sizes 3} is minimized when $K^{1/2}2^{-t}=\frac{t^{-1/2}}{2\log 2}$, and so we pick $t$ to be the smallest positive integer such that $K^{1/2}2^{-t}\leq\frac{t^{-1/2}}{2\log 2}$.
With this, we continue \eqref{eq.partition across sizes 2}:
\begin{align}
\nonumber
\bigg|\bigg\langle \sum_{i\in\mathcal{I}}x_i\varphi_i,\sum_{j\in\mathcal{J}}y_j\varphi_j \bigg\rangle\bigg|
&<\hat\theta\Big(2(t^{1/2}+K^{1/2}2^{-t})\Big)^2\\
\label{eq.partition across sizes 4}
&\leq 4\hat\theta \bigg(t^{1/2}+\frac{t^{-1/2}}{2\log 2}\bigg)^2
=4\hat\theta\bigg(t+\frac{1}{\log 2}+\frac{1}{(2\log 2)^2t}\bigg).
\end{align}
From here, we claim that $t\leq\lceil\frac{\log K}{\log 2}\rceil$.
Considering the definition of $t$, this is easily verified for $K=2,3,\ldots,7$ by showing $K^{1/2}2^{-s}\leq\frac{s^{-1/2}}{2\log 2}$ for $s=\lceil\frac{\log K}{\log 2}\rceil$.
For $K\geq8$, one can use calculus to verify the second inequality of the following:
\begin{equation*}
K^{1/2}2^{-\lceil\frac{\log K}{\log 2}\rceil}
\leq K^{1/2}2^{-\frac{\log K}{\log 2}}
\leq \frac{1}{2\log 2}\bigg(\frac{\log K}{\log 2}+1\bigg)^{-1/2}
\leq \frac{1}{2\log 2}\bigg\lceil\frac{\log K}{\log 2}\bigg\rceil^{-1/2},
\end{equation*}
meaning $t\leq\lceil\frac{\log K}{\log 2}\rceil$.
Substituting $t\leq\frac{\log K}{\log 2}+1$ and $t\geq1$ into \eqref{eq.partition across sizes 4} then gives
\begin{equation*}
\bigg|\bigg\langle \sum_{i\in\mathcal{I}}x_i\varphi_i,\sum_{j\in\mathcal{J}}y_j\varphi_j \bigg\rangle\bigg|
<4\hat\theta\bigg(\frac{\log K}{\log2}+1+\frac{1}{\log 2}+\frac{1}{(2\log 2)^2}\bigg)\\
=\hat\theta(C_0\log K+C_1),
\end{equation*}
with $C_0\approx 5.77$, $C_1\approx 11.85$.
As such, \eqref{eq.to bound} is $\leq C'\hat\theta\log K$ with $C'=C_0+\frac{C_1}{\log 2}$ in this case.

We are now ready for the final bound on \eqref{eq.to bound} in which we apply no constraints on the $x_i$'s and $y_j$'s.
To do this, we consider the positive and negative real and imaginary parts of these coefficients:
\begin{equation*}
x_i=\sum_{k=0}^3x_{i,k}\mathrm{i}^k\quad \mbox{s.t.}\quad x_{i,k}\geq0\quad \forall k,
\end{equation*}
and similarly for the $y_j$'s.
With this decomposition, we apply the triangle inequality to get
\begin{align*}
\bigg|\bigg\langle \sum_{i\in\mathcal{I}}x_i\varphi_i,\sum_{j\in\mathcal{J}}y_j\varphi_j \bigg\rangle\bigg|
&=\bigg|\bigg\langle \sum_{i\in\mathcal{I}}\sum_{k_1=0}^3x_{i,k_1}\mathrm{i}^{k_1}\varphi_i,\sum_{j\in\mathcal{J}}\sum_{k_2=0}^3y_{j,k_2}\mathrm{i}^{k_2}\varphi_j \bigg\rangle\bigg|\\
&\leq\sum_{k_1=0}^3\sum_{k_2=0}^3\bigg|\bigg\langle \sum_{i\in\mathcal{I}}x_{i,k_1}\varphi_i,\sum_{j\in\mathcal{J}}y_{j,k_2}\varphi_j \bigg\rangle\bigg|.
\end{align*}
Finally, we normalize the coefficients by $(\sum_{i\in\mathcal{I}}x_{i,k_1}^2)^{1/2}$ and $(\sum_{j\in\mathcal{J}}y_{j,k_2}^2)^{1/2}$ so we can apply our second bound:
\begin{align*}
\bigg|\bigg\langle \sum_{i\in\mathcal{I}}x_i\varphi_i,\sum_{j\in\mathcal{J}}y_j\varphi_j \bigg\rangle\bigg|
&\leq\sum_{k_1=0}^3\sum_{k_2=0}^3\bigg(\sum_{i\in\mathcal{I}}x_{i,k_1}^2\bigg)^{1/2}\bigg(\sum_{j\in\mathcal{J}}y_{j,k_2}^2\bigg)^{1/2} C'\hat\theta\log K\\
&\leq(C\hat\theta\log K)\|x\|\|y\|,
\end{align*}
where $C=4C'\approx74.17$ by the Cauchy-Schwarz inequality, and so we are done.
\end{proof}


\begin{thebibliography}{WW}

\bibitem{AlexeevCM:arxiv11}
B.~Alexeev, J.~Cahill, D.G.~Mixon,
Full spark frames,
submitted, Available online: arXiv:1110.3548

\bibitem{AlonN:06}
N.~Alon, A.~Naor,
Approximating the cut-norm via Grothendieck's inequality,
SIAM J.~Comput.~35 (2006) 787--803.

\bibitem{ApplebaumHSC:09}
L.~Applebaum, S.D.~Howard, S.~Searle and R.~Calderbank,
Chirp sensing codes: Deterministic compressed sensing measurements for fast recovery,
Appl.~Comp.~Harmon.~Anal.~26 (2009) 283--290.

\bibitem{BaraniukDDW:08}
R.~Baraniuk, M.~Davenport, R.~DeVore and M.~Wakin,
A simple proof of the restricted isometry property for random matrices,
Constr.~Approx.~28 (2008) 253--263.

\bibitem{Bernstein:46}
S.N.~Bernstein,
Theory of Probability, 4th ed., Moscow-Leningrad, 1946.

\bibitem{Bollobas:01}
B.~Bollob\'{a}s,
Random Graphs, 2nd ed., Cambridge, 2001.

\bibitem{BourgainDFKK:11}
J.~Bourgain, S.~Dilworth, K.~Ford, S.~Konyagin and D.~Kutzarova,
Explicit constructions of RIP matrices and related problems,
Duke Math.~J.~159 (2011) 145--185.

\bibitem{Candes:08}
E.J.~Cand\`{e}s,
The restricted isometry property and its implications for compressed sensing,
C.~R.~Acad.~Sci.~Paris, Ser.~I 346 (2008) 589--592.

\bibitem{CandesT:05}
E.J.~Cand\`{e}s and T.~Tao,
Decoding by linear programming,
IEEE~Trans.~Inform.~Theory 44 (2005) 4203--4215.

\bibitem{CandesT:07}
E.J.~Cand\`{e}s and T.~Tao, 
The Dantzig selector: Statistical estimation when $p$ is much larger than $n$,
Ann. Statist. 35 (2007) 2313--2351.

\bibitem{ChungGW:89}
F.R.K.~Chung, R.L.~Graham and R.M.~Wilson,
Quasi-random graphs,
Combinat.~9 (1989) 345--362.

\bibitem{Cohen:88}
S.D.~Cohen, 
Clique numbers of Paley graphs, 
Quaestiones Math.~11 (1988), 225--231.

\bibitem{DavidsonS:01}
K.R.~Davidson and S.J.~Szarek, 
Local operator theory, random matrices and Banach spaces, 
In: Handbook in Banach Spaces Vol I, ed. W.B.~Johnson, J.~Lindenstrauss,
Elsevier (2001), 317--366.

\bibitem{DeVore:07}
R.A.~DeVore, 
Deterministic constructions of compressed sensing matrices, 
J.~Complexity 23 (2007) 918--925.

\bibitem{DohonoE:03}
D.L.~Donoho and M.~Elad,
Optimally sparse representation in general (nonorthogonal) dictionaries via $\ell_1$ minimization,
Proc.~Nat.~Acad.~Sci.~USA 100 (2003) 2197--2202.

\bibitem{FickusMT:10}
M.~Fickus, D.G.~Mixon and J.C.~Tremain,
Steiner equiangular tight frames, 
Linear Algebra Appl.~436 (2012) 1014--1027.

\bibitem{Gerschgorin:31}
S.~Gerschgorin,
\"{U}ber die Abgrenzung der Eigenwerte einer Matrix,
Izv.~Akad.~Nauk.~USSR~Otd.~Fiz.-Mat. 7 (1931) 749--754.

\bibitem{GrahamR:90}
S.W.~Graham and C.J.~Ringrose, 
Lower bounds for least quadratic non-residues, 
Prog.~Math.~85 (1990) 269--309.

\bibitem{HooryLW:06}
S.~Hoory, N.~Linial and A.~Wigderson,
Expander graphs and their applications,
Bull.~Amer.~Math.~Soc.~43 (2006) 439--561.

\bibitem{LaurentM:00}
B.~Laurent and P.~Massart,
Adaptive estimation of a quadratic functional by model selection,
Ann.~Statist.~28 (2000) 1302--1338.

\bibitem{Natarajan:95}
B.K.~Natarajan,
Sparse approximate solutions to linear systems,
SIAM~J.~Comput.~24 (1995) 227--234.

\bibitem{Peralta:92}
R.~Peralta,
On the distribution of quadratic residues and nonresidues modulo a prime number,
Math.~Comput.~58 (1992) 433--440.

\bibitem{Rauhut:08}
H.~Rauhut,
Stability results for random sampling of sparse trigonometric polynomials,
IEEE Trans.~Inform.~Theory 54 (2008) 5661--5670.

\bibitem{Renes:07}
J.M.~Renes,
Equiangular tight frames from Paley tournaments,
Linear Algebra Appl.~426 (2007) 497--501.

\bibitem{RudelsonV:07}
M.~Rudelson and R.~Vershynin,
On sparse reconstruction from Fourier and Gaussian measurements,
Comm.~Pure Appl.~Math.~61 (2008) 1025--1045.

\bibitem{Sachs:62}
H.~Sachs,
\"{U}ber selbstkomplement\"{a}re Graphen,
Publ.~Math.~Debrecen 9 (1962) 270--288. 

\bibitem{Seidel:73}
J.J.~Seidel, 
A survey of two-graphs, 
in: Proc.~Intern.~Coll.~Teorie Combinatorie (1973) 481--511.

\bibitem{StevenhagenL:96}
P.~Stevenhagen, H.W.~Lenstra,
Chebotar\"{e}v and his density theorem, 
Math.~Intelligencer 18 (1996) 26--37.

\bibitem{StrohmerH:03}
T.~Strohmer and R.W.~Heath,
Grassmannian frames with applications to coding and communication,
Appl.~Comp.~Harmon.~Anal. 14 (2003) 257--275.

\bibitem{Tao:07}
T.~Tao,
Open question: Deterministic UUP matrices, \\
\texttt{http://terrytao.wordpress.com/2007/07/02/open-question-deterministic-uup-matrices}.

\bibitem{vanLintS:66}
J.H.~van Lint, J.J.~Seidel, 
Equilateral point sets in elliptic geometry, 
Nederl.~Akad.~Wetensch.~Proc.~Ser.~A 69 (1966) 335--348; Indag.~Math.~28.

\bibitem{Waldron:09}
S.~Waldron,
On the construction of equiangular frames from graphs,
Linear Algebra Appl.~431 (2009) 2228--2242.

\bibitem{Welch:74}
L.R.~Welch,
Lower bounds on the maximum cross correlation of signals,
IEEE Trans.~Inform.~Theory 20 (1974) 397--399.

\bibitem{XiaZG:05}
P.~Xia, S.~Zhou and G.B~Giannakis,
Achieving the Welch bound with difference sets,
IEEE Trans.~Inform.~Theory 51 (2005) 1900--1907.

\bibitem{Yurinskii:76}
V.V.~Yurinskii,
Exponential inequalities for sums of random vectors,
J.~Multivariate Anal.~6 (1976) 473--499.

\end{thebibliography}
\end{document}